\documentclass[a4paper,preprint,3p,11pt,authoryear]{elsarticle}

\usepackage{amssymb}
\usepackage{amsmath}
\usepackage{amsthm}

\usepackage[colorlinks=true, allcolors=black, hypertexnames=false]{hyperref}
\usepackage{algorithm}
\usepackage{algpseudocode}
\usepackage{multirow}
\usepackage{booktabs}
\usepackage{threeparttable}
\usepackage{setspace}
\usepackage{tikz-cd}
\usepackage{enumitem}
\def\dd{{\rm d}}
\def\ii{{\rm i}}
\def\ee{{\rm e}}

\def\mb{\mathbb}

\def\toP{\stackrel{\mb{P}}{\to}}

\def\lf{\lfloor}
\def\rf{\rfloor}

\numberwithin{equation}{section}
\newtheorem{theorem}{Theorem}
\newtheorem{lemma}[theorem]{Lemma}
\newtheorem{proposition}[theorem]{Proposition}

\newtheorem{assumption}{Assumption}
\theoremstyle{definition}

\theoremstyle{remark}
\newtheorem{remark}{Remark}

\journal{European Journal of Operational Research}

\begin{document}

\begin{frontmatter}

%% Title, authors and addresses

%% use the tnoteref command within \title for footnotes;
%% use the tnotetext command for theassociated footnote;
%% use the fnref command within \author or \affiliation for footnotes;
%% use the fntext command for theassociated footnote;
%% use the corref command within \author for corresponding author footnotes;
%% use the cortext command for theassociated footnote;
%% use the ead command for the email address,
%% and the form \ead[url] for the home page:
%% \title{Title\tnoteref{label1}}
%% \tnotetext[label1]{}
%% \author{Name\corref{cor1}\fnref{label2}}
%% \ead{email address}
%% \ead[url]{home page}
%% \fntext[label2]{}
%% \cortext[cor1]{}
%% \affiliation{organization={},
%%            addressline={}, 
%%            city={},
%%            postcode={}, 
%%            state={},
%%            country={}}
%% \fntext[label3]{}

\title{Long memory score-driven models as approximations for rough Ornstein-Uhlenbeck processes\tnoteref{label2}} %% Article title
\tnotetext[label2]{This work is partially supported by the National Natural Science Foundation of China (Grant No. 12371144) and Fundamental Research Funds for the Central Universities (Grant No. CXJJ-2024-439).}

%% use optional labels to link authors explicitly to addresses:
\author[label1]{Yinhao Wu\corref{cor1}}
\cortext[cor1]{Corresponding author}
\ead{yinhaowu@stu.sufe.edu.cn}
\author[label1]{Ping He}
\ead{pinghe@mail.shufe.edu.cn}
\affiliation[label1]{organization={School of Mathematics, Shanghai University of Finance and Economics},%Department and Organization
            city={Shanghai},
            postcode={200433}, 
            country={China}}

%% Abstract
\begin{abstract}
%% Text of abstract
This paper investigates the continuous-time limit of score-driven models with long memory. By extending score-driven models to incorporate infinite-lag structures with coefficients exhibiting heavy-tailed decay, we establish their weak convergence, under appropriate scaling, to fractional Ornstein-Uhlenbeck processes with Hurst parameter $H < 1/2$. When score-driven models are used to characterize the dynamics of volatility, they serve as discrete-time approximations for rough volatility. We present several examples, including EGARCH($\infty$) whose limits give rise to a new class of rough volatility models. Building on this framework, we develop a GPU-accelerated Monte Carlo scheme based on divide-and-conquer fast Fourier transforms, reducing the computational complexity for each path from $\mathcal{O}(N^2)$ to $\mathcal{O}(N\log^2 N)$ for large-scale rough volatility simulation and option pricing.
\end{abstract}

%% Keywords

\end{frontmatter}

\section{Introduction}
Score-driven models, also known as generalized autoregressive score models, were proposed by \cite{creal2013generalized} and \cite{harvey2013dynamic} as a class of observation-driven time series models in which parameter updates are driven by the {\it score}---the gradient of the log-likelihood function. By directly linking the score of the data distribution to parameter dynamics in a unified framework, these models inherit flexibility and generality, encompassing many classical models such as GARCH \citep{bollerslev1986generalized} and ACD \citep{engle1998autoregressive} models. Theoretically, score-driven models have been proven optimal in minimizing the Kullback-Leibler divergence between the true distribution and the postulated distribution \citep{gorgi2023optimality}. Moreover, in statistical inference, the model's closed-form likelihood function facilitates efficient maximum likelihood estimation, exhibiting lower computational costs compared to parameter-driven models. In application, they also demonstrate robust empirical performance across diverse fields such as economics, finance and management (see, e.g., \cite{barunik2016modeling, sarlo2023lumpy, hansen2026option}). Within just a few years, the literature on SD models has grown to encompass over 400 publications and the comprehensive bibliography available at \href{http://www.gasmodel.com/gaspapers.htm}{\texttt{gasmodel.com}}.

Given the remarkable influence of score-driven models in time series analysis, a natural question arises: what would their continuous time counterparts look like? Building on the method of \cite{nelson1990arch} for deriving the continuous time limit of GARCH, \cite{buccheri2021continuous} have investigated continuous time limit of score-driven volatility models, obtaining a two dimensional diffusion process. Similarly, \cite{wu2024continuous} investigated the continuous time limit of the quasi score-driven volatility models proposed by \cite{blasques2023quasi}, and derived a correlated stochastic volatility model. However, the score-driven models they studied only incorporate first order or finite order lags, which are all Markovian. In fact, when the model was first introduced in \cite{creal2013generalized}, the question was raised as to whether it could be extended to an infinite lag form. \cite{janus2014long} and \cite{opschoor2019fractional} pursued such extensions by proposing long memory score-driven models. This naturally leads to the question: What would be the continuous time counterpart of such a long memory version?

On the other hand, since \cite{gatheral2018volatility} discovered that volatility is rough, rough volatility models have gained popularity in financial modeling. These models are characterized by the use of fractional Brownian motion with a Hurst parameter $H<1/2$, whose paths are rougher than those of Brownian motion, hence the name. \cite{el2018microstructural} investigated the microstructure of rough volatility by employing Hawkes processes to model the arrival dynamics in limit order books. This theoretical basis originates from the discovery of \cite{jaisson2016rough} that the scaling limit of nearly unstable heavy tailed Hawkes processes is a rough Cox-Ingersoll-Ross (CIR) process. Their self exciting structure parallels the observation driven feedback in score-driven models, and the slow decay of their kernels generates long memory. In particular, \cite{cai2026scaling} studied the INAR($\infty$) process, a discrete time counterpart of the Hawkes process, and obtained similar convergence results. Indeed, from a score-driven perspective, and as discussed in Section \ref{sec:long}, a Poisson score-driven update with inverse Fisher scaling also recovers the conditional intensity representation of a Poisson INAR($\infty$) model once a full-history lag structure is introduced. This motivates our investigation into the scaling limit of score-driven models when extended to incorporate long memory.

This paper establishes that a sequence of long memory score-driven models, under appropriate scaling, converges weakly to a rough Ornstein-Uhlenbeck (OU) process. Unlike the cumulative limits in \citet{jaisson2016rough} and \citet{wang2025rough}, we prove convergence of the rescaled score-driven parameter process itself. By modeling time-varying parameters as log-volatility (as in EGARCH-type models), we obtain as the limit a class of rough volatility models. Consequently, long memory score-driven volatility models can serve as discrete-time approximations for such rough volatility models, facilitating efficient Monte Carlo simulations for financial applications. Indeed, there exists a growing body of literature dedicated to approximating rough volatility models, such as Markovian approximation of the fractional kernel \citep{ma2022fast, bayer2024efficient} and continuous-time Markov chain (CTMC) methods \citep{yang2025general}. However, to the best of our knowledge, apart from \cite{wang2025rough}'s approach using INAR($\infty$) processes to approximate the rough Heston model, this study presents one of the few methodologies employing time series models to approximate rough volatility processes. At the same time, a substantial parallel literature addresses fractional integration in volatility via ARFIMA-type time series models; see, for example, \cite{shi2023volatility}, who explore the boundary between long memory and rough volatility in the context of the volatility puzzle. Therefore, we also present specific examples of long memory score-driven volatility models to further demonstrate the practical versatility and broad applicability of our theoretical results. Under the same heavy tailed and nearly unstable scaling, long memory score-driven volatility models offer a time series explanation for rough volatility: in the EGARCH type examples in Section \ref{sec:examples}, the macroscopic limit of the observation driven volatility dynamics is a rough log volatility process.

The remainder of the paper is structured as follows. Section \ref{sec:long} introduces the long memory score-driven models that we focus on and provides a heuristic derivation of their continuous time limits. In Section \ref{sec:main results}, we present our main convergence results. Proofs of the theorem and propositions are given in the main text when brief; all remaining proofs are deferred to the Appendix and supplementary material. Section \ref{sec:examples} applies these results to volatility modeling, illustrating with the Gamma-GED-EGARCH($\infty$) model and its extension. We conduct corresponding approximations of rough volatility models and option pricing. Section \ref{sec:conclusion} concludes our findings and provides an outlook for future research.

Notation. We use lowercase $\lambda$ for the discrete score-driven parameter and uppercase $\Lambda$ for its rescaled continuous time version and its limit. The lag weights are denoted by $\phi_i$; in the nearly unstable sequence, $\phi_i^{(n)}=a_n\phi_i$, and $\psi^{(n)}$ is the associated renewal sequence. The symbol $\alpha$ denotes the tail index of the weights and the roughness parameter, with $H=\alpha-1/2$. When H\"{o}lder spaces are used, their regularity exponent is denoted by $\eta$.

\section{Long memory score-driven models}\label{sec:long}

Since \citet{engle1982autoregressive} discovered that economic data exhibit conditional heteroskedasticity, it has gradually been recognized that certain characteristics of time series are not immutable. Specifically, let $\{y_n\}_{n\in \mb{N}}$ be a time series and $\mathcal{F}_{n}=\sigma(y_n, y_{n-1},\ldots, y_0)$ be the $\sigma$-field it generates. We model the conditional density of $y_n$ given
\[y_n|\mathcal{F}_{n-1}\sim p(y_n|\lambda_n, \mathcal{F}_{n-1}),\]
with a time varying parameter $\lambda_n$. To address this issue, \citet{cox1981statistical} proposed two modeling approaches: observation driven models, where the dynamics of the time varying parameters are governed by past observations of the series itself, and parameter driven models, in which the parameter dynamics are driven by another stochastic process. 

The score-driven model is a class of observation driven models whose key idea is to update time varying parameters based on the score of the observation. Specifically, the parameter update follows the form: 
\begin{equation}\label{lambda1}
    \lambda_{n+1}=\omega+\beta\lambda_{n}+A S(\lambda_n)\nabla_{n},
\end{equation}
where $\nabla_{n}=\frac{\partial \log p(y_n|\lambda_n, \mathcal{F}_{n-1})}{\partial \lambda_n}$ denotes the partial derivative of the log density function of observations with respect to the parameter, commonly referred to as the (Fisher) score in statistics, and $A$ is a scalar coefficient. Consequently, these models are aptly termed score-driven models. The function $S(\cdot)$ is known as the scaling function, which accounts for the curvature of the log density function. It is typically chosen as the negative exponent of the conditional Fisher information, i.e., $S(\lambda_n)=[\mb{E}(\nabla_n^2|\mathcal{F}_{n-1})]^{-a}$, with common choices for $a$ being $0$, $1/2$, or $1$. To illustrate, when the density function $p(\cdot)$ is chosen to be that of a normal distribution and the parameter $\lambda$ is identified as the conditional variance $\sigma^2$, with $a = 1$, the score-driven model reduces to the classical GARCH model.

The equation given by \eqref{lambda1} represents a first order score-driven model, as the relationship between the score and the parameter is lagged by only one period. This framework can be extended to higher orders, such as $p$-th order or even infinite order, as in \citet{janus2014long}, 
\[\lambda_{n}=\omega+\sum_{i=1}^{n}\phi_{i} S(\lambda_{n-i})\nabla_{n-i}.\]
It should be noted that our time index starts from zero; thus, the term ``infinite'' here refers to dependence on the entire history, from time 0 up to the present.
This formulation appears similar to the ARCH($\infty$) model, but the key difference lies in the fact that the lagged score term $S \nabla$ in the score-driven model is a martingale difference, whereas the $y^2$ term in the ARCH($\infty$) model is not. In other words, rewriting the ARCH($\infty$) model in a form driven by martingale differences gives
\[
\sigma^2_{n} = \omega + \sum_{i=1}^{n} \phi_{i} y^2_{n-i} = \omega + \sum_{i=1}^{n} \phi_{i} \sigma^2_{n-i} + \sum_{i=1}^{n} \phi_{i} (y^2_{n-i} - \sigma^2_{n-i}),
\]
which includes an additional infinite autoregressive component. A useful precedent is the FIGAS construction of \citet{opschoor2019fractional}. In their notation, $V_t$ denotes the conditional covariance matrix and $s_t$ is the scaled score. Setting $\tilde{s}_t=V_t+s_t$, their FIGAS recursion can be rewritten so that $V_{t+1}$ is driven by an infinite weighted sum of current and past $\tilde{s}_t$. Following this idea in our setting, define
\[
    z_m:=\lambda_m+S(\lambda_m)\nabla_m.
\]
Applying a common lag filter to the lagged values $z_{n-i}$ gives

\begin{equation}\label{lambdamain}
    \lambda_{n} = \omega + \sum_{i=1}^{n} \phi_{i} [\lambda_{n-i} + S(\lambda_{n-i}) \nabla_{n-i}].
\end{equation}

In \eqref{lambdamain}, the same lag filter acts on the lagged pseudo-observations $z_{n-i}$. The term $\lambda_{n-i}$ gives the autoregressive feedback, while the scaled score term is a martingale difference under the predictive score property and predictable scaling.

\begin{remark}[Poisson score and INAR conditional intensity]
The relation with INAR-type dynamics can be seen in the Poisson case. If $N_n|\mathcal{F}_{n-1}\sim\mathrm{Poisson}(\lambda_n)$, then the raw score is $\nabla_n=N_n/\lambda_n-1$. Under inverse Fisher scaling, $S(\lambda_n)=\lambda_n$, so the scaled score is $S(\lambda_n)\nabla_n=N_n-\lambda_n$ and hence $z_n=\lambda_n+S(\lambda_n)\nabla_n=N_n$. In this case, \eqref{lambdamain} becomes $\lambda_n=\omega+\sum_{i=1}^n\phi_i N_{n-i}$, with $N_n|\mathcal{F}_{n-1}\sim\mathrm{Poisson}(\lambda_n)$. This coincides with the conditional intensity representation of a univariate Poisson INAR($\infty$) model, as in \citet{wang2025rough}, with $\omega$ as the immigration rate and $\phi_i$ as reproduction coefficients.
\end{remark}

The long memory score-driven model considered henceforth is based on \eqref{lambdamain}. To investigate its scaling limit, we define a sequence of processes:
\begin{equation}
    \begin{split}
        &y_t^{(n)}|\mathcal{F}_{t-1}^{(n)} \sim p\left(y_t^{(n)} | \lambda_t^{(n)}, \mathcal{F}^{(n)}_{t-1}\right), \\
        &\lambda_{t}^{(n)} = \omega_n + \sum_{i=1}^{t} \phi_i^{(n)} \left[\lambda_{t-i}^{(n)} + S(\lambda_{t-i}^{(n)}) \nabla_{t-i}^{(n)}\right],
    \end{split}
\end{equation}
where $\phi_i^{(n)} = a_n \phi_i>0$. Let $a_n>0$ and $a_n\uparrow 1$ as $n\to\infty$, $\|\phi\|=\sum_{i=1}^{\infty} \phi_i= 1$, so that $\|\phi^{(n)}\| \uparrow 1$. This case is referred to as nearly unstable. Specifically, we set $\phi_i \sim K/i^{1+\alpha}$, exhibiting power law decay, where $\alpha \in (1/2, 1)$, a form referred to as ``heavy tailed'' in \citet{jaisson2016rough}. This power law condition is critical for the limit to generate roughness and compatible with the fractional filters used in \citet{opschoor2019fractional}. In their notation, with $L$ denoting the lag operator, the FIGAS lag polynomial can be written as
\[
    W_{\rm OL}(L)
    =1-\frac{(1-L)^d(1-\varphi_{\rm OL}L)}{1-B_{\rm OL}L}
    =\sum_{i\ge1}w_i^{\rm OL}L^i .
\]
By the binomial expansion of $(1-L)^d$, if $d\in(0,1)$, $|B_{\rm OL}|<1$, and $|\varphi_{\rm OL}|<1$, then
\[
    w_i^{\rm OL}\sim
    \frac{d(1-\varphi_{\rm OL})}{(1-B_{\rm OL})\Gamma(1-d)}\,i^{-(1+d)} .
\]
Thus $d$ plays the role of $\alpha$ in our notation: the FIGAS coefficients provide a parametric example of the power law tail condition imposed on $\{\phi_i\}$, which is stated directly at the level of the coefficient asymptotics.

\begin{remark}
The term ``long memory'' in this paper is used in the filtering sense: the recursion uses the available history through power law lag weights. This usage is consistent with the hyperbolically weighted FIGAS filters of \citet{opschoor2019fractional}; it is also consistent with the heavy tailed INAR($\infty$) kernels of \citet{wang2025rough}, where such kernels create the long memory that produces rough volatility.
\end{remark}

 The $\sigma$-field is defined as $\mathcal{F}^{(n)}_t = \sigma(\lambda_0^{(n)}, y_i^{(n)}, \lambda_{i+1}^{(n)}, i = 0, 1, \ldots, t)$, since $\lambda_{t+1}^{(n)}$ is determined when $\{y_i^{(n)}\}_{i \leq t}$ are known.

Our primary focus is on the dynamics of the time varying parameter $\lambda_t^{(n)}$. Before stating the rigorous result, we give a heuristic derivation that explains why a rough Volterra limit is expected. For notational simplicity, the superscript $n$ is sometimes omitted when no confusion can arise; the terms $a_n$ and $\omega_n$ are kept to indicate the sequence indexed by $n$. Define $M_t = \sum_{i=0}^{t} S(\lambda_{i}) \nabla_{i}$, with $M_{-1} = 0$. Clearly, $M$ is a martingale, since
\[
\mb{E}(M_{t} - M_{t-1} | \mathcal{F}_{t-1}) = \mb{E}[S(\lambda_{t}) \nabla_{t} | \mathcal{F}_{t-1}] = S(\lambda_{t}) \mb{E}[\nabla_{t} | \mathcal{F}_{t-1}] = 0.
\]
Then, 
\[\begin{split}
    \lambda_t&=\omega_n+\sum_{i=0}^{t-1}a_n\phi_{t-i} [\lambda_{i}+S(\lambda_{i})\nabla_{i}]\\
    &=\omega_n+\sum_{i=0}^{t-1}a_n\phi_{t-i}(M_{i}-M_{i-1})+\sum_{i=0}^{t-1}a_n\phi_{t-i} \lambda_{i}\\
    &=\omega_n+\sum_{i=0}^{t-1}a_n\phi_{t-i}(M_{i}-M_{i-1})+\sum_{i=0}^{t-1}\psi_{t-i}^{(n)} \left[\omega_n+\sum_{j=0}^{i-1}a_n\phi_{i-j}(M_{j}-M_{j-1})\right].
\end{split}\]
The final equality follows from the transformation between AR($\infty$) and MA($\infty$), that is, 
\[\lambda_t=\varepsilon_t+\sum_{i=0}^{t-1}a_n\phi_{t-i}\lambda_i\quad\Longleftrightarrow\quad\lambda_t=\varepsilon_t+\sum_{i=0}^{t-1}\psi^{(n)}_{t-i}\varepsilon_i. \]
Here, $\psi^{(n)} = \sum_{k=1}^{\infty} (a_n \phi)^{*k}$ represents the sum of the $k$ fold convolution of $\phi^{(n)}$. By rearranging the expression and exchanging the order of the double summation, we obtain
\[\lambda_t=\omega_n+\sum_{i=0}^{t-1}\psi_{t-i}^{(n)}\omega_n+\sum_{i=0}^{t-1}a_n\phi_{t-i}(M_{i}-M_{i-1})+\sum_{j=0}^{t-2}(M_j-M_{j-1})\sum_{i=j+1}^{t-1}a_n\phi_{i-j}\psi_{t-i}^{(n)}.\]
Note that 
\[\sum_{i=j+1}^{t-1}a_n\phi_{i-j}\psi_{t-i}^{(n)}=(\phi^{(n)}*\psi^{(n)})_{t-j}, \]
and 
\[\phi^{(n)}*\psi^{(n)}=\sum_{k=1}^{\infty}(\phi^{(n)})^{*{k+1}}=\sum_{k=2}^{\infty}(\phi^{(n)})^{*{k}}=\psi^{(n)}-\phi^{(n)},\]
we have 
\begin{equation}\label{lambdatrans}
    \begin{split}
        \lambda_t&=\omega_n+\sum_{i=0}^{t-1}\psi_{t-i}^{(n)}\omega_n+a_n \phi_1(M_{t-1}-M_{t-2})+\sum_{j=0}^{t-2}\psi_{t-j}^{(n)}(M_j-M_{j-1})\\
        &=\omega_n+\sum_{i=0}^{t-1}\psi_{t-i}^{(n)}\omega_n+\sum_{j=0}^{t-1}\psi_{t-j}^{(n)}(M_j-M_{j-1}).
    \end{split}
\end{equation}

Since $\mb{E}[\lambda_t^{(n)}] \to \omega_n/(1 - a_n)$ as $t \to \infty$, we apply a rescaling to $\lambda_t^{(n)}$ and define a new process
\[\Lambda_t^{(n)} := \frac{(1 - a_n)\theta}{\omega_n} \lambda_{\lf nt \rf}^{(n)}, \quad t \in [0,1],\]
where $\theta\in\mathbb R\setminus\{0\}$ is a constant controlling the mean level of $\Lambda_t^{(n)}$. We aim to investigate the limiting behavior of $\Lambda_t^{(n)}$ as $n \to \infty$.

Let $\Delta M_k=M_{k}-M_{k-1}$ and write $\mathrm{Var}[\Delta M_k|\mathcal{F}_{k-1}] = S^2(\lambda_k)\mathbb{E}[(\nabla_k)^2|\mathcal{F}_{k-1}] =: U^2(\lambda_k)$, defining the function $U:\mathbb{R}\to\mathbb{R}_{\ge 0}$ implicitly through this conditional variance. It follows from \eqref{lambdatrans} that 
\begin{equation}\label{Lambda_main}
    \begin{split}
        \Lambda_t^{(n)}&=(1-a_n)\theta+(1-a_n)\theta\sum_{i=0}^{\lf nt\rf-1}\psi_{\lf nt\rf-i}^{(n)}+\frac{1-a_n}{\omega_n}\theta\sum_{j=0}^{\lf nt\rf-1}\psi_{\lf nt\rf-j}^{(n)}\Delta M_j^{(n)}\\
        &=(1-a_n)\theta+\theta\sum_{i=1}^{\lf nt\rf}(1-a_n)\psi_{i}^{(n)}+\frac{\theta}{\sqrt{n}\omega_n}\sum_{j=0}^{\lf nt\rf-1}n(1-a_n)\psi_{\lf nt\rf-j}^{(n)}U(\lambda_{j}^{(n)})\frac{\Delta M_j^{(n)}}{U(\lambda_{j}^{(n)})\sqrt{n}}.
    \end{split}
\end{equation}
The representation in \eqref{Lambda_main} shows the two ingredients that drive the limit. First, the scaled score martingale gives the stochastic input. Formally, if the martingale array satisfies a functional central limit theorem, then one expects
\[
\sum_{j=0}^{\lf nt \rf - 1} \frac{\Delta M_j^{(n)}}{U(\lambda_{j}^{(n)}) \sqrt{n}} \Rightarrow W_t,
\]
where $W=(W_t)_{t\geq 0}$ is a standard Brownian motion. Second, the renewal sequence associated with the nearly unstable heavy tailed lag weights produces the Volterra kernel. According to \citet{cai2026scaling}, for $1/2 < \alpha < 1$, when $1-a_n$ tends to zero at rate $n^{-\alpha}$ and the coefficients $\phi_i \sim K/i^{1+\alpha}$ exhibit power law decay, the following weak convergence holds:
\begin{equation}\label{mlweak}
    n (1 - a_n) \psi^{(n)}_{\lf nx \rf} \to f^{\alpha,\kappa}(x) := \kappa x^{\alpha-1} E_{\alpha,\alpha}(-\kappa x^\alpha), \quad \mbox{as} \ n \to \infty,
\end{equation}
where $\kappa > 0$ is a constant, and
\[
E_{\alpha,\beta}(x) = \sum_{n=0}^\infty \frac{x^n}{\Gamma(\alpha n + \beta)}
\]
is the Mittag-Leffler function. A property of the function is that it decays to zero as $x \to -\infty$, and $E_{\alpha,\alpha}(0) = \frac{1}{\Gamma(\alpha)}$. Thus, in terms of singularity, the behavior of $f^{\alpha,\kappa}(t-s)$ is analogous to that of the Riemann-Liouville kernel $(t-s)^{\alpha-1}$. This is the source of roughness: since $\alpha-1\in(-1/2,0)$, the stochastic convolution with this kernel has the roughness index $H=\alpha-1/2<1/2$. Equation \eqref{Lambda_main} also shows how the score scale enters the diffusion coefficient. Keeping the normalization in \eqref{Lambda_main} and the renewal scaling in \eqref{mlweak}, a nondegenerate formal limit with a state dependent diffusion coefficient would require, for fixed $y$,
\[
    \frac{\theta}{\sqrt n\,\omega_n}
    U\left(\frac{\omega_n}{(1-a_n)\theta}y\right)
    \to \bar U(y).
\]
The corresponding candidate Volterra equation is
\begin{equation}\label{formal_general_volterra}
    \Lambda_t
    =
    \theta\int_0^t f^{\alpha,\kappa}(s)\dd s
    +
    \int_0^t f^{\alpha,\kappa}(t-s)\bar U(\Lambda_s)\dd W_s .
\end{equation}

For instance, if $U(x)\sim\sigma |x|^qL(|x|)$ with $0<q<1$ and slowly varying $L$, this amounts to a condition on the speed of $\omega_n$:
\[
    \frac{\theta}{\sqrt n\,\omega_n}
    \left|\frac{\omega_n}{(1-a_n)\theta}\right|^q
    L\left(\left|\frac{\omega_n}{(1-a_n)\theta}\right|\right)
    \to \ell,
\]
which formally gives $\bar U(y)=\sigma\ell |y|^q$. When $q=1/2$, $L\equiv1$, and the limit is nonnegative, this gives the rough CIR case \citep{jaisson2016rough}.

When $U(x)\to\gamma$, the same condition reduces to $\theta/(\sqrt n\,\omega_n)\to\mu$ and $\bar U(y)=\mu\gamma$. Under this asymptotically constant score scale, the formal limit in \eqref{Lambda_main} becomes
\[
    \Lambda_t
    =
    \theta\int_0^t f^{\alpha,\kappa}(s)\dd s
    +
    \mu\gamma\int_0^t f^{\alpha,\kappa}(t-s)\dd W_s.
\]
In fact, this is precisely the rough OU process \eqref{roughou} proved in Theorem \ref{mainthm}; the two equations have the same unique solution.

\section{Main results}\label{sec:main results}

Following the heuristic derivation above, in this section, we state our main theorem and present its proof. We begin with the following assumptions.

\begin{assumption}\label{a1}
    The function $U:\mathbb R\to(0,\infty)$ is locally bounded and locally bounded away from zero, and satisfies $\lim_{|x|\to\infty}U(x)=\gamma$ for some $\gamma>0$.
\end{assumption}

\begin{assumption}\label{a2}
The sequences $a_n$ and $\omega_n$ satisfy, for some $\kappa\in(0,\infty)$ and $\mu\in\mathbb R\setminus\{0\}$,
    \[
    \frac{\alpha n^\alpha(1-a_n)}{K\Gamma(1-\alpha)}\to\kappa,
    \qquad
    \frac{\theta}{\sqrt n\,\omega_n}\to\mu.
    \]
\end{assumption}

\begin{assumption}\label{a3}
    The martingale difference arrays $\{\Delta M_k^{(n)}:0\le k\le n,\ n\geq 1\}$ satisfy: for every $p > 2$, 
    \[
    \sup_{n\geq 1}\max_{0\le k\le n}\mathbb{E}\left|\Delta M_k^{(n)}\right|^p < \infty.
    \]
\end{assumption}

\begin{remark}

Assumption \ref{a1} concerns the conditional standard deviation of the score-driven martingale increment. Recall that
\[
    U^2(\lambda_k)
    =
    S^2(\lambda_k)
    \mb E\!\left[(\nabla_k)^2\mid\mathcal F_{k-1}\right].
\]
Thus $U(x)\to\gamma$ means that the conditional score scale becomes asymptotically constant along the scaling regime. This is immediate when $S$ is chosen as the $-1/2$ power of the conditional Fisher information, and it can be checked directly in the GED and Gaussian quasi score examples in Section \ref{sec:examples}. Assumption \ref{a2} also has a time scale interpretation. We view the fixed horizon as divided into $n$ intervals of length $1/n$; in financial applications, increasing $n$ can therefore be understood as observing the process at progressively higher frequencies. Since $\sum_{i\geq1}\phi_i^{(n)}=a_n$, the limit $a_n\uparrow1$ means that the feedback from past parameter values is almost fully retained from one increasingly short interval to the next. For the macroscopic scaling studied here, the rate $1-a_n\asymp n^{-\alpha}$ is required for the heavy tailed renewal sequence to have a nondegenerate limit on the fixed horizon, while $\omega_n\asymp n^{-1/2}$ places the accumulated score innovations at the diffusive scale. Assumption \ref{a3} gives the moment control used in the martingale functional limit theorem and the H\"older tightness argument. Requiring moments of all orders in Assumption \ref{a3} covers the full range $\alpha\in(1/2,1)$; if only a finite $p$-th moment exists for some $p > 2$, the theorem still holds for $\alpha \in (\frac{p+2}{2p}, 1)$.

\end{remark}

Then, we have the following theorem. 
\begin{theorem}\label{mainthm}
    Under Assumptions \ref{a1}-\ref{a3}, the rescaled long memory score-driven parameter process
    \[\Lambda_t^{(n)}:=\frac{1-a_n}{\omega_n}\theta\lambda^{(n)}_{\lf nt\rf}, \quad t\in[0,1]\]
    converges weakly in the Skorokhod space $\mathcal{D}([0,1])$ to a rough fractional diffusion process given by the following SDE:
    \begin{equation}\label{roughou}
        \Lambda_t=\frac{1}{\Gamma(\alpha)}\int_{0}^t (t-s)^{\alpha-1}\kappa(\theta-\Lambda_s)\dd s+\frac{\gamma\kappa\mu}{\Gamma(\alpha)}\int_0^t (t-s)^{\alpha-1}\dd W_s. 
    \end{equation}
    Here, $\kappa$ and $\mu$ are the constants in Assumption \ref{a2}.
\end{theorem}

In fact, we restrict our attention to the rough OU case, namely the constant diffusion case of the formal Volterra equation \eqref{formal_general_volterra}. Unlike the cumulative limits in \citet{jaisson2016rough} and \citet{wang2025rough}, we prove convergence of the rescaled parameter process itself. A direct Hawkes intensity result based on a dedicated $C$-tightness criterion is given by \citet{horst2023convergence}. Here, GFO continuity gives convergence of a continuous auxiliary representation, while Proposition \ref{indistinguishable} shows uniform asymptotic indistinguishability from $\Lambda^{(n)}$, transferring tightness and weak convergence to the original parameter process.

To this end, we regard the rough diffusion process as the image of a diffusion process under the generalized fractional operator (GFO), which is inspired by \cite{horvath2024functional}. Specifically, for a H\"{o}lder exponent $\eta\in (0,1)$ and $\alpha\in (-\eta, 1-\eta)$, $\mathcal{G}^\alpha$ is the GFO associated to kernel 
\[g\in \mathcal{L}^{\alpha}:=\left\{g\in C^2((0,1]): \left|\frac{g(u)}{u^\alpha}\right|, \left|\frac{g'(u)}{u^{\alpha-1}}\right| \ \text{and}\  \left|\frac{g''(u)}{u^{\alpha-2}}\right| \ \mbox{are bounded}\right\}\]
defined on the $\eta$ H\"{o}lder space $\mathcal{C}^\eta ([0,1])$\footnote{Throughout this paper, we use the simplified notation $\mathcal{C}^{\eta}([0,1])$ to denote the H\"{o}lder space typically written as $C^{0,\eta}([0,1])$, where the superscript 0 indicates that no differentiability is assumed beyond H\"{o}lder continuity.},
\begin{equation}\label{GFO_def}
    (\mathcal{G}^\alpha f)(t)=\left\{\begin{aligned}
        &\int_{0}^t (f(s)-f(0))\frac{\dd }{\dd t}g(t-s)\dd s,\quad \text{if} \ \alpha\geq 0,\\
        &\frac{\dd }{\dd t}\int_{0}^t (f(s)-f(0))g(t-s)\dd s,\quad \text{if} \ \alpha< 0.
    \end{aligned}\right.
\end{equation}
When $g(x) = \frac{x^\alpha}{\Gamma(\alpha+1)}$, this reduces to the classical Riemann-Liouville operator. In this paper, we choose $g(x)=f^{\alpha, \kappa}(x)\in \mathcal{L}^{\alpha-1}$ and focus on the case $\alpha\in(1/2,1)$. Consequently, the corresponding GFO actually becomes
\begin{equation}\label{GFO}
    (\mathcal{G}^{\alpha-1} f)(t)=\frac{\dd }{\dd t}\int_{0}^t (f(s)-f(0))f^{\alpha, \kappa}(t-s)\dd s.
\end{equation}

Define $Y_t:=\theta t+\mu \gamma W_t$, and let
\begin{equation}\label{Lambda_t}
    \Lambda_t :=\theta\int_{0}^t f^{\alpha,\kappa}(s)\dd s+\mu\gamma\int_0^t f^{\alpha,\kappa}(t-s)\dd W_s
\end{equation}
The process $\Lambda$ is the solution of the equation \eqref{roughou} we aim to converge to:
\begin{proposition}\label{same_solution}
    The process $\Lambda$ is the unique strong solution to the following rough SDE: 
    \[\Lambda_t=\frac{1}{\Gamma(\alpha)}\int_{0}^t (t-s)^{\alpha-1}\kappa(\theta-\Lambda_s)\dd s+\frac{\gamma\kappa\mu}{\Gamma(\alpha)}\int_0^t (t-s)^{\alpha-1}\dd W_s. \]
\end{proposition}

\begin{proof}
    See Section S1 of the supplementary material. 
\end{proof}

\begin{remark}[Equivalent representations of the limiting process]
The preceding construction relates the three forms of the limiting process used below. Theorem \ref{mainthm} states the rough Volterra equation, while Proposition \ref{same_solution} shows that the Mittag-Leffler convolution representation in \eqref{Lambda_t} is its unique strong solution. Moreover, with $Y_t=\theta t+\mu\gamma W_t$, we will show in Proposition \ref{Gaction} that the same process can be written as $\Lambda=\mathcal{G}^{\alpha-1}Y$, where $\mathcal{G}^{\alpha-1}$ is the GFO in \eqref{GFO} with kernel $f^{\alpha,\kappa}$. This GFO representation is used to prove tightness and apply the continuous mapping theorem. Since the noise kernel behaves near the origin like $(t-s)^{\alpha-1}$, the corresponding roughness index is $H=\alpha-1/2\in(0,1/2)$.
\end{remark}

Accordingly, in the subsequent proof, we take the clearer form of equation \eqref{Lambda_t} as the target of convergence. Consider the partition $\mathcal{T}:=\{t_i=i/n, i=0,1,\ldots, n\}$. We then define
\begin{align}
    &Y_t^{(n)} = \theta t + \frac{\theta}{n\omega_n} \left[ (1 - nt+\lf nt\rf) M_{\lf nt \rf-1}^{(n)} + (nt-\lf nt\rf) M_{\lf nt \rf}^{(n)} \right],\\
    &\widetilde{\Lambda}_t^{(n)}=\sum_{i=0}^{\lf nt\rf-1}\left(\theta + \frac{\theta}{\omega_n} \Delta M_{i}^{(n)}\right)\int_{t_i}^{t_{i+1}}f^{\alpha,\kappa}(t-s)\dd s+\left(\theta + \frac{\theta}{\omega_n} \Delta M_{\lf nt\rf}^{(n)}\right)\int_{t_{\lf nt\rf}}^{t}f^{\alpha,\kappa}(t-s)\dd s. \label{tlidelambda}
\end{align}

The subsequent proof proceeds in four steps.
\begin{enumerate}[label=(\roman*)]
    \item Establishing the representations $\Lambda_t = (\mathcal{G}^{\alpha-1}Y)_t$, and $\widetilde{\Lambda}_t^{(n)} = (\mathcal{G}^{\alpha-1}Y^{(n)})_t$;
    \item Proving the weak convergence $Y^{(n)}\Rightarrow Y$ in H\"{o}lder space $\mathcal{C}^{\eta}$, with $\eta\in(1-\alpha,1/2)$ chosen for the GFO step;
    \item Demonstrating that $\mathcal{G}^{\alpha-1}$ is a continuous operator from $\mathcal{C}^{\eta}$ to $\mathcal{C}^{\eta+\alpha-1}$, which yields the convergence $\widetilde{\Lambda}_t^{(n)} \Rightarrow \Lambda_t$ in $\mathcal{C}^{\eta+\alpha-1}$, and thus in $\mathcal{D}$;
    \item Finally, by proving $\widetilde{\Lambda}_t^{(n)}$ and $\Lambda_t^{(n)}$ are asymptotically indistinguishable, we obtain $\Lambda_t^{(n)}\Rightarrow \Lambda_t$. 
\end{enumerate}

\begin{center}
\begin{tikzcd}[column sep=huge, row sep=huge]
  & Y^{(n)}_t \arrow[r, Rightarrow, "\|\cdot\|_{\eta}"] \arrow[d, "\mathcal{G}^{\alpha-1}"] 
    & Y_t \arrow[d, "\mathcal{G}^{\alpha-1}"] 
    & (\mathcal{C}^{\eta}([0,1]), \|\cdot\|_{\eta}) \arrow[d, "\mathcal{G}^{\alpha-1}"] \\
  \Lambda^{(n)}_t \arrow[r, leftrightarrow, "\approx"] 
    & \widetilde{\Lambda}^{(n)}_t \arrow[r, Rightarrow, "\|\cdot\|_{\eta+\alpha-1}"] 
    & \Lambda_t 
    & (\mathcal{C}^{\eta+\alpha-1}([0,1]), \|\cdot\|_{\eta+\alpha-1})
\end{tikzcd}
\end{center}

The first step consists of two representation proofs, which are relatively straightforward and are presented immediately after the corresponding propositions. The proof of the second step is relegated to \ref{proof:weak_converge}. The third step follows from the GFO continuity result of \citet{horvath2024functional}. Finally, as the proof of the fourth step is technical, it is available in the supplementary material.
 Before proceeding with the main proof, we first recall the key convergence result \eqref{mlweak} for the Mittag-Leffler function, and extend this result to $L^2$ convergence for our subsequent arguments. 

\subsection{\texorpdfstring{Convergence to $f^{\alpha,\kappa}$}{Convergence to f alpha kappa}}

\cite{jaisson2016rough} prove that the mean of geometric sums converges in distribution to a random variable with density function $f^{\alpha,\kappa}(x)$. Our setting differs slightly because the summed terms follow a discrete distribution.

Specifically, we define the sequence of random variables
\[
    G_n=\frac{\sum_{j=1}^{I_n}U_j}{n},
\]
where $I_n$ follows a geometric distribution with parameter $1-a_n$, that is, $\mb P(I_n=\ell)=(1-a_n)a_n^{\ell-1}$ for $\ell\ge1$, and $\{U_j:j\ge1\}$ is a sequence of i.i.d. random variables taking values in $\mb N^+$ with $\mb P(U_1=i)=\phi_i$ for $i=1,2,\ldots$. Then the random variable $G_n$, called the mean of geometric sums, has its law given by
\[
\begin{split}
    \mb P\left(G_n=\frac{k}{n}\right)
    &=
    \sum_{\ell\ge1}
    \mb P\left(\sum_{j=1}^{\ell}U_j=k\right)
    \mb P(I_n=\ell)  \\
    &=
    (1-a_n)\sum_{\ell\ge1}a_n^{\ell-1}\phi^{*\ell}_k
    =
    \frac{1-a_n}{a_n}\psi_k^{(n)}
    =:p_{n,k},\qquad k\ge1.
\end{split}
\]
Thus the law of the mean of geometric sums is determined by $p_{n,k}$, and its distribution function is
\[
    F_n(x)=\sum_{k=1}^{\lfloor nx\rfloor}p_{n,k}
    =
    \sum_{k=1}^{\lfloor nx\rfloor}
      \frac{(1-a_n)\psi_k^{(n)}}{a_n}.
\]
To discuss convergence at the density level, we use the associated histogram density
\begin{equation}\label{rho_histogram}
    \rho_n(x)
    :=
    \sum_{k\ge 1}
    n p_{n,k}\mathbf 1_{((k-1)/n,k/n]}(x)
    =
    \sum_{k\ge 1}
    \frac{n(1-a_n)}{a_n}\psi_k^{(n)}
    \mathbf 1_{((k-1)/n,k/n]}(x).
\end{equation}
Let $\widehat \phi(v):=\sum_{k\ge 1}\phi_k \ee^{-\ii kv}$ be the characteristic function of $U_1$ under our Fourier convention. Considering the characteristic function of $G_n$, we obtain
\begin{equation}\label{geometric_cf}
\begin{split}
    \chi_n(u)
    &:=
    \mb{E}\left[\ee^{-\ii uG_n}\right]
    =
    \sum_{\ell\ge1}
    \mb{E}\left[
    \ee^{-\frac{\ii u}{n}\sum_{j=1}^{\ell}U_j}
    \right]\mb P(I_n=\ell)\\
    &=
    (1-a_n)\sum_{\ell\ge1}a_n^{\ell-1}\widehat\phi(u/n)^\ell
    =
    \frac{(1-a_n)\widehat\phi(u/n)}
         {1-a_n\widehat\phi(u/n)}.
\end{split}
\end{equation}
For the histogram density \eqref{rho_histogram}, integrating over each interval gives the Fourier transform
\begin{equation}\label{rhohat_histogram}
    \widehat \rho_n(u)
    =
    \ee^{\ii u/n}
    m_n(u)\,
    \frac{(1-a_n)\widehat\phi(u/n)}
         {1-a_n\widehat\phi(u/n)},
    \qquad
    m_n(u):=\frac{1-\ee^{-\ii u/n}}{\ii u/n},
\end{equation}
with the convention $m_n(0)=1$. Here $\chi_n$ is the characteristic function of $G_n$, whereas $\widehat\rho_n$ is the Fourier transform of the histogram density.

Since $\phi_k\sim K k^{-1-\alpha}$, the Karamata-Tauberian theorem yields
\begin{equation}\label{hatphi}
    1-\widehat \phi(u)
    \underset{u\to 0}{\sim}
    \frac{K\Gamma(1-\alpha)}{\alpha}(\ii u)^\alpha .
\end{equation}
By Assumption \ref{a2}, $\kappa\in(0,\infty)$.
It follows from \eqref{geometric_cf} that, for every fixed $u\in\mb{R}$,
\[
    \chi_n(u)\to
    \frac{\kappa}{\kappa+(\ii u)^\alpha}.
\]
The limiting function is the Fourier transform of $f^{\alpha,\kappa}(x)=\kappa x^{\alpha-1}E_{\alpha,\alpha}(-\kappa x^\alpha)$, extended by zero to the negative half line. The continuity theorem for characteristic functions therefore gives the weak convergence of $G_n$ to the distribution with density $f^{\alpha,\kappa}$. Since the limiting distribution function is continuous, the convergence of the distribution functions is uniform:
\begin{equation}\label{uniform}
    \sup_{x\ge 0}\left|
    \sum_{k=1}^{\lfloor nx\rfloor}
    \frac{(1-a_n)\psi_k^{(n)}}{a_n}
    -
    \int_0^x f^{\alpha,\kappa}(s)\dd s
    \right|\to 0.
\end{equation}

Moreover, since $\ee^{\ii u/n}m_n(u)\to 1$, the same pointwise limit holds for $\widehat\rho_n(u)$. The factor $m_n$ in \eqref{rhohat_histogram} comes from averaging over each interval and is used in the $L^2$ estimate. We record the required estimates for $\widehat\rho_n$ in the following lemma.

\begin{lemma}\label{philemma}
    There exist constants $C>0$ and $\delta\in(0,\pi)$ such that, for all sufficiently large $n$,
    \begin{equation}\label{rhohat_low_main}
        |\widehat \rho_n(u)|
        \le
        \frac{C}{1+|u|^\alpha},
        \qquad |u|\le \delta n,
    \end{equation}
    and
    \begin{equation}\label{rhohat_high_main}
        \int_{|u|>\delta n}|\widehat\rho_n(u)|^2\dd u
        \le
        Cn(1-a_n)^2\to 0.
    \end{equation}
    For every fixed $u\in\mb{R}$,
    \begin{equation}\label{rhohat_pointwise_main}
        \widehat\rho_n(u)\to
        \widehat f^{\alpha,\kappa}(u)
        =
        \frac{\kappa}{\kappa+(\ii u)^\alpha}.
    \end{equation}
\end{lemma}
\begin{proof}
    See Section S2 of the supplementary material.
\end{proof}

Together with Plancherel's identity, Lemma \ref{philemma} gives the required $L^2$ convergence.

\begin{proposition}\label{L2convergence}
    The histogram density $\rho_n$ defined in \eqref{rho_histogram} converges to $f^{\alpha,\kappa}$ in $L^2([0,1])$:
    \begin{equation}\label{rho_L2}
        \|\rho_n-f^{\alpha,\kappa}\|_{L^2([0,1])}
        :=
        \left(\int_0^1
        \left[\rho_n(x)-f^{\alpha,\kappa}(x)\right]^2\dd x\right)^{1/2}
        \to 0 .
    \end{equation}
\end{proposition}
\begin{proof}
    See Section S3 of the supplementary material.
\end{proof}

\subsection{Proof of main theorem}
We proceed with the proof following the aforementioned four steps. First, we need to demonstrate how the convergence target $\Lambda_t$ can be expressed in terms of the operator $\mathcal{G}^{\alpha-1}$ defined in \eqref{GFO}. 
\begin{proposition}\label{Gaction}
    The equality $\Lambda_t=(\mathcal{G}^{\alpha-1} Y)_t$ holds almost surely for all $t\in [0,1]$. 
\end{proposition}

\begin{proof}
    Since the paths of $Y_t$ have $1/2-\epsilon$ H\"{o}lder continuity, for $\alpha\in(1/2,1)$, the exponent $\alpha-1$ falls within $(-1/2,0)$, which is well defined in this context of GFO. By definition \eqref{GFO}, 
    \[(\mathcal{G}^{\alpha-1} Y)(t)=\frac{\dd }{\dd t}\int_0^t (Y(s)-Y(0)) f^{\alpha,\kappa}(t-s)\dd s.\]
    For $\varepsilon>0$, define the operator 
    \[(\mathcal{G}^\alpha_{\varepsilon} Y)(t)= \int_0^{t-\varepsilon} (Y(s)-Y(0)) f^{\alpha,\kappa}(t-s)\dd s.\]
    Then, for any $t\in[0,1]$, we almost surely have 
    \[\begin{split}
        &\frac{\dd }{\dd t}(\mathcal{G}^\alpha_{\varepsilon} Y)(t)=(Y(t-\varepsilon)-Y(0))f^{\alpha, \kappa}(\varepsilon)+\int_0^{t-\varepsilon}(Y(s)-Y(0)) (f^{\alpha,\kappa})'(t-s)\dd s\\
        &=\int_{0}^{t-\varepsilon}f^{\alpha,\kappa}(t-s) \dd Y_s=\theta\int_{0}^{t-\varepsilon}f^{\alpha,\kappa}(t-s)\dd s+\mu\gamma \int_{0}^{t-\varepsilon}f^{\alpha,\kappa}(t-s)\dd W_s.\\
    \end{split}\]
    As $\varepsilon \to 0$, the drift part converges by the integrability of $f^{\alpha,\kappa}$, and the stochastic convolution converges in $L^2$ by the $L^2$ integrability of $f^{\alpha,\kappa}$ \citep{horst2023convergence}. Integrating the same estimates over $t$ gives convergence in $L^2(\Omega\times[0,1])$; together with the convergence of $\mathcal G^\alpha_\varepsilon Y$, this identifies the weak derivative of $\mathcal G^\alpha Y$, yielding
    \[(\mathcal{G}^{\alpha-1}Y)(t)=\lim_{\varepsilon\to0 }\frac{\dd }{\dd t}(\mathcal{G}^\alpha_{\varepsilon} Y)(t)=\Lambda_t. \]
    The identity holds almost surely at every rational time. Since both sides have continuous paths, it extends almost surely to all $t\in[0,1].$
\end{proof}

\begin{proposition}
    The equality $ \widetilde{\Lambda}_t^{(n)}=(\mathcal{G}^{\alpha-1} Y^{(n)})_t$ holds almost surely for all $t\in [0,1]$.
\end{proposition}

\begin{proof}
    It should be noted that $Y^{(n)}$ is actually a rescaled version of the martingale $M^{(n)}$. To ensure its paths belong to a H\"{o}lder space, we adopt the linear interpolation
    \[Y_t^{(n)} = \theta t + \frac{\theta}{n\omega_n} \left[ (1 - nt+\lf nt\rf) M_{\lf nt \rf-1}^{(n)} + (nt-\lf nt\rf) M_{\lf nt \rf}^{(n)} \right],\]
    which is piecewise differentiable, and for $t \in (t_i, t_{i+1})$, it holds that
    \begin{equation}\label{Y_deriv}
        \frac{\dd Y^{(n)}(t)}{\dd t}=\theta + \frac{\theta}{\omega_n} \Delta M_{i}^{(n)}.
    \end{equation}
    Since the paths of $Y^{(n)}$ are piecewise linear, they are Lipschitz continuous on $[0,1]$. Consequently, the GFO is well defined on $Y_t^{(n)}$. By the definition of the GFO in \eqref{GFO} and the absolute continuity of $Y^{(n)}$, we can express the operator via the convolution of the kernel and the derivative: 
    \[\begin{split}
        &(\mathcal{G}^{\alpha-1} Y^{(n)})(t) = \frac{\dd}{\dd t} \int_0^t (Y^{(n)}_s - Y^{(n)}_0) f^{\alpha,\kappa}(t-s) \dd s = \int_0^t f^{\alpha,\kappa}(t-s) \frac{\dd Y^{(n)}_s}{\dd s} \dd s\\
        &= \sum_{i=0}^{\lf nt\rf-1}\left(\theta + \frac{\theta}{\omega_n} \Delta M_{i}^{(n)}\right)\int_{t_i}^{t_{i+1}}f^{\alpha,\kappa}(t-s)\dd s +\left(\theta + \frac{\theta}{\omega_n} \Delta M_{\lf nt\rf}^{(n)}\right)\int_{t_{\lf nt\rf}}^{t}f^{\alpha,\kappa}(t-s)\dd s.
    \end{split}
    \]
    In the last step, we split the integral over the grid points $\mathcal{T}$ and substitute the explicit form of the derivative from \eqref{Y_deriv}, which matches the definition of $\widetilde{\Lambda}_t^{(n)}$.
\end{proof}

In the second step, we prove that $Y^{(n)}\Rightarrow Y$ in H\"{o}lder space $\mathcal{C}^{\eta}$ for $\eta\in(0,1/2)$; when applying the GFO below, we choose $\eta\in(1-\alpha,1/2)$. The Donsker invariance principle for convergence in H\"{o}lder spaces was first established by \cite{lamperti1962convergence}. \cite{ravckauskas2004necessary} provided necessary and sufficient conditions for the convergence of i.i.d. partial sum processes in H\"{o}lder spaces, while results for martingale difference sequences were derived by \cite{giraudo2016holderian} from a dynamical systems perspective. However, these results concern only the scaling limits of a single sequence and do not directly apply to triangular arrays. The FCLT for martingale difference arrays is well established in the Skorokhod space \citep{hall2014martingale}. Consequently, once tightness in the H\"{o}lder space is verified, the desired convergence in this space follows immediately. 

\begin{proposition}\label{weak_converge}
    For any $\eta\in(0,1/2)$, $Y^{(n)}$ converges weakly to $Y$ in H\"{o}lder space $\mathcal{C}^{\eta}([0,1])$.
\end{proposition}

\begin{proof}
    See \ref{proof:weak_converge}.
\end{proof}

Subsequently, we establish the continuity of the operator $\mathcal{G}^{\alpha-1}$, thereby enabling the application of the continuous mapping theorem. Indeed, this result is encompassed within Proposition 2.2 of \cite{horvath2024functional}, where a proof of continuity for general GFOs is provided. We directly apply their result to obtain the following proposition. 

\begin{proposition}[Continuity of the GFO]\label{prop:gfo_continuity}
    For any $\alpha \in (1/2, 1)$ and any $\eta \in (1 - \alpha, 1/2)$, the operator $\mathcal{G}^{\alpha-1}$ associated with the kernel $f^{\alpha,\kappa}\in\mathcal{L}^{\alpha-1}$ is continuous from $\mathcal{C}^\eta([0,1])$ to $\mathcal{C}^{\eta+\alpha-1}([0,1])$.
\end{proposition}

This continuity, together with $Y^{(n)} \Rightarrow Y$ in the same $\mathcal{C}^{\eta}$ space, gives
\[
    \mathcal{G}^{\alpha-1}Y^{(n)}
    \Rightarrow
    \mathcal{G}^{\alpha-1}Y
\]
by the continuous mapping theorem. From the above representations, this entails
\[\widetilde{\Lambda}_t^{(n)}\Rightarrow \Lambda_t \ \text{in}\ \mathcal{C}^{\eta+\alpha-1}([0,1]). \]
According to Lemma 3.10 in \cite{horvath2024functional}, we can readily extend this weak convergence to the continuous function space $\mathcal{C}([0,1])$ and Skorokhod space $\mathcal{D}([0,1])$. 

We have established the convergence of $\widetilde{\Lambda}_t^{(n)}$. However, our primary process of interest is $\Lambda_t^{(n)}$ given by \eqref{Lambda_main}, with $\widetilde{\Lambda}_t^{(n)}$ serving merely as an auxiliary process. Finally, we will demonstrate that the two processes are asymptotically indistinguishable.

\begin{proposition}\label{indistinguishable}
    As $n\to\infty$, 
    \[\sup_{t\in[0,1]}\left|\Lambda_t^{(n)}-\widetilde{\Lambda}_t^{(n)}\right|\toP 0.\]
\end{proposition}

\begin{proof}
    See Section S4 of the supplementary material. 
\end{proof}

Note that the convergence in Proposition \ref{indistinguishable} is stated in the uniform metric. Since the Skorokhod metric satisfies $d_J(f,g)\le\|f-g\|_\infty$ for all $f,g\in\mathcal{D}([0,1])$, the uniform convergence in probability implies
    \[d_J\!\left(\Lambda^{(n)},\widetilde{\Lambda}^{(n)}\right)\toP 0, \]
which means $\Lambda_t^{(n)}$ and $\widetilde{\Lambda}_t^{(n)}$ are asymptotically indistinguishable in $\mathcal{D}([0,1])$. By virtue of Theorem 3.1 in \cite{billingsley2013convergence}, it follows that $\Lambda_t^{(n)} \Rightarrow \Lambda_t$.

\section{Specific examples for volatility models}\label{sec:examples}
\subsection{Gamma-GED-EGARCH($\infty$) model}
When we use the score-driven model to characterize volatility, we focus on the following observed time series:
\[y_n=\sqrt{\varphi(\lambda_n)}\varepsilon_n, \quad \varepsilon_n|\mathcal{F}_{n-1}\stackrel{d}{\sim}\text{density} \ f(\cdot).\]
If $\varepsilon_n$ has unit variance, $\sqrt{\varphi(\lambda_n)}$ is in fact the conditional volatility of $y_n$. Here, the function $\varphi: \mb{R}\to\mb{R}^+$ is monotonic and differentiable, referred to as the link function. Moreover, if $\varphi$ is the identity mapping (restricted to $\mb{R}^+$), then the time varying parameter $\lambda_n$ directly characterizes the conditional variance of $y_n$. According to \cite{buccheri2021continuous}, the score in this case can be expressed as 
\[\nabla_n=\frac{-\varphi'(\lambda_n)}{2\varphi(\lambda_n)}\left[1+\frac{f'\left(\frac{y_n}{\sqrt{\varphi(\lambda_n)}}\right)}{f\left(\frac{y_n}{\sqrt{\varphi(\lambda_n)}}\right)}\frac{y_n}{\sqrt{\varphi(\lambda_n)}}\right].\]

In this section, we set $\varphi(x)=\ee^{2x}$, which implies that $\lambda_n:=\ln \sigma_n$ is the log volatility. We assume that $f$ is the density of the Generalized Error Distribution (GED) with parameter $\nu$, that is
\[f(x)=\frac{1}{2^{1+1/\nu}\Gamma(1+1/\nu)}\exp\left(-\frac{|x|^\nu}{2}\right).\]
In this case, the score is given by
\[\nabla_n=\frac{\nu}{2}\left|\frac{y_n}{\ee^{\lambda_n}}\right|^\nu-1=\frac{\nu}{2}|\varepsilon_n|^\nu-1.\]
Since the Fisher information is a constant, we set $S(\lambda_n)=1$. Thus, the dynamics of $\lambda_n$ are given by
\[\lambda_{n} = \omega + \sum_{i=1}^{n} \phi_{i} \left[\lambda_{n-i} + \frac{\nu}{2}|\varepsilon_{n-i}|^\nu-1\right].\]
Following the terminology in \cite{harvey2017volatility}, we refer to this model as the Gamma-GED-EGARCH($\infty$). When $\nu=2$, the GED reduces to a standard normal distribution, and this model corresponds to the EGARCH model in \cite{nelson1991conditional} but without asymmetry. Based on the previous definitions, we define the following sequence of scaled processes:
\begin{equation}\label{gammaged}
    \begin{split}
        &X_t^{(n)}=\frac{1}{\sqrt{\chi n}}\sum_{i=1}^{\lf nt\rf}\left[\exp(\lambda_i^{(n)})\right]^{\frac{1-a_n}{\omega_n}\theta}\varepsilon_i, \quad \varepsilon_i|\mathcal{F}_{i-1}\stackrel{d}{\sim} \ \text{GED}(\nu),\\
        &\Lambda_t^{(n)}:=\frac{1-a_n}{\omega_n}\theta\lambda^{(n)}_{\lf nt\rf}, \quad \lambda_{t}^{(n)} = \omega_n + \sum_{i=1}^{t} \phi_i^{(n)} \left[\lambda_{t-i}^{(n)} + \frac{\nu}{2}|\varepsilon_{t-i}|^\nu-1\right],
    \end{split}
\end{equation}
where $\phi_i\sim K/i^{1+\alpha}$, $a_n$ and $\omega_n$ satisfy Assumption \ref{a2}, and $\chi =\mb{E}(\varepsilon^2)=2^{2/\nu}\frac{\Gamma(3/\nu)}{\Gamma(1/\nu)}$.

The martingale increment $\Delta M_k=\nu|\varepsilon_k|^\nu/2-1$ has moments of every order and conditional variance $\nu$. Thus Assumptions \ref{a1}--\ref{a3} hold under the scaling in \eqref{gammaged}, with $\gamma=\sqrt{\nu}$.

\begin{theorem}\label{example_theorem}
    For any $\alpha\in(1/2,1)$, as $n\to\infty$, the pair $(X_t^{(n)},\Lambda_t^{(n)})$ defined in \eqref{gammaged} converges weakly to $(X_t,\Lambda_t)$ in the Skorokhod space $\mathcal{D}([0,1],\mathbb{R}^2)$, 
    \begin{equation}\label{roughvol}
        \begin{split}
        &X_t=\int_{0}^t \ee^{\Lambda_s} \dd B_s,\\
        &\Lambda_t=\frac{1}{\Gamma(\alpha)}\int_{0}^t (t-s)^{\alpha-1}\kappa(\theta-\Lambda_s)\dd s+\frac{\sqrt{\nu}\kappa\mu}{\Gamma(\alpha)}\int_0^t (t-s)^{\alpha-1}\dd W_s, 
        \end{split}
    \end{equation}
    where $B$ and $W$ are independent standard Brownian motions, and $\kappa,\mu$ are as in Assumption \ref{a2}.
\end{theorem}
\begin{proof}
    See Section S5 of the supplementary material. 
\end{proof}

\subsection{An extension as rough volatility approximations}
To use \eqref{roughvol} as a rough volatility approximation with nonzero initial level and leverage, two features are needed. First, the initial value of $(X_t, \Lambda_t)$ is zero in \eqref{roughvol}; second, the two Brownian motions are independent, which prevents the model from capturing the leverage effect in the market. 

For the first issue, we adopt the strategy in \cite{wang2025rough}, adding a baseline in the dynamics of $\lambda_t^{(n)}$, that is
\[\lambda_{t}^{(n)} = \omega_t^{(n)} + \sum_{i=1}^{t} \phi_i^{(n)} \left[\lambda_{t-i}^{(n)} + \frac{\nu}{2}|\varepsilon_{t-i}|^\nu-1\right],\]
where
\[\omega_t^{(n)}=\omega_n+\xi\omega_n\left(\frac{1}{1-a_n}\left(1-\sum_{s=1}^{t-1}\phi_s^{(n)}\right)-\sum_{s=1}^{t-1}\phi_s^{(n)}\right),\]
with $\xi>0$. This results in the following limit for $\Lambda_t$: 
\[\Lambda_t=\theta\xi+\frac{1}{\Gamma(\alpha)}\int_{0}^t (t-s)^{\alpha-1}\kappa(\theta-\Lambda_s)\dd s+\frac{\sqrt{\nu}\kappa\mu}{\Gamma(\alpha)}\int_0^t (t-s)^{\alpha-1}\dd W_s.\]

For the second issue, \cite{wu2024continuous} find that the quasi score of \cite{blasques2023quasi} can generate correlation between the two Brownian motions in the limit. Therefore, we can use the score of an asymmetric distribution. We set $\varepsilon\sim\text{GED}(2)=\mathcal{N}(0,1)$ for simplicity in the Monte Carlo simulation\footnote{The resulting limit has the same form as \eqref{roughvol}, up to a constant diffusion coefficient.}, and use the score of density $q(\lambda, y)=\ee^{-\lambda/2}g(y\ee^{-\lambda/2})$, where
\[g(x)=\frac{(\rho+\zeta)^2}{2}|x|\ee^{-\rho x-\zeta|x|}.\]
In the framework of quasi score driven model, it recovers the asymmetric structure analogous to the EGARCH model, that is
\[\lambda_{t}^{(n)} = \omega_n + \sum_{i=1}^{t} \phi_i^{(n)} \left[\lambda_{t-i}^{(n)} + \rho \varepsilon_{t-i}+\zeta \left(|\varepsilon_{t-i}|-\sqrt{2/\pi}\right)\right].\]

In summary, we consider the following sequence of scaled quasi score driven long memory volatility models,
\begin{equation}\label{qsdlong}
    \begin{split}
        &X_t^{(n)}=X_0-\frac{1}{2n}\sum_{i=1}^{\lf nt\rf}\left[\exp(2\lambda_i^{(n)})\right]^{\frac{1-a_n}{\omega_n}\theta}+\frac{1}{\sqrt{n}}\sum_{i=1}^{\lf nt\rf}\left[\exp(\lambda_i^{(n)})\right]^{\frac{1-a_n}{\omega_n}\theta}\varepsilon_i, \quad \varepsilon_i|\mathcal{F}_{i-1}\stackrel{d}{\sim} \ \mathcal{N}(0,1),\\
        &\Lambda_t^{(n)}:=\frac{1-a_n}{\omega_n}\theta\lambda^{(n)}_{\lf nt\rf\vee1},\\
        &\lambda_{t}^{(n)} = \omega_t^{(n)} + \sum_{i=1}^{t} \phi_i^{(n)} \left[\lambda_{t-i}^{(n)} + \rho \varepsilon_{t-i}+\zeta \left(|\varepsilon_{t-i}|-\sqrt{2/\pi}\right)\right].
    \end{split}
\end{equation}
Let $v:=\sqrt{\rho^2+\zeta^2(1-2/\pi)}$ and $\bar\rho:=\rho/v$. The quasi score increment has variance $v^2$ and moments of every order. Thus Assumptions \ref{a1}--\ref{a3} hold under the chosen scaling, with $\gamma=v$. The following theorem establishes the limit of \eqref{qsdlong}. 
\begin{theorem}[Rough volatility with leverage]\label{newroughvol}
    Under Assumptions \ref{a1}-\ref{a3}, as $n\to\infty$, the pair $(X_t^{(n)},\Lambda_t^{(n)})$ defined in \eqref{qsdlong} converges weakly in $\mathcal{D}([0,1],\mathbb{R}^2)$ to $(X_t,\Lambda_t)$, 
    \begin{equation}\label{eq: newroughvol}
    \begin{split}
    X_t &= X_0-\frac{1}{2}\int_0^t \ee^{2\Lambda_s} \dd s
    +\int_0^t \ee^{\Lambda_s}\dd B_s, \\
    \Lambda_t &= \theta\xi+\frac{1}{\Gamma(\alpha)}\int_0^t(t-s)^{\alpha-1}\kappa(\theta-\Lambda_s) \dd s
    +\frac{v\kappa\mu}{\Gamma(\alpha)}\int_0^t(t-s)^{\alpha-1} \dd W_s,
    \end{split}
    \end{equation}
where $B$ and $W$ are standard Brownian motions with $\langle B,W\rangle_t=\bar\rho t$. 
\end{theorem}

\begin{proof}
    See Section S6 of the supplementary material. 
\end{proof}

The limit \eqref{eq: newroughvol} is a rough volatility model, where $\Lambda_t$ characterizes the log volatility in the log asset price dynamics of $X_t$. Although the log volatility process $\Lambda_t$ is driven by additive fractional noise, the actual volatility $\sigma_t = \ee^{\Lambda_t}$ exhibits multiplicative, state dependent dynamics. This structure is precisely analogous to the Rough Fractional Stochastic Volatility (RFSV) framework of \cite{gatheral2018volatility} and the Rough Bergomi model of \cite{bayer2016pricing}, both of which characterize log volatility (or log variance) via a process driven by additive fractional noise. Furthermore, in Equation \eqref{eq: newroughvol}, the kernel $(t-s)^{\alpha-1}$ emerges not only in the noise but also in the drift term. In this regard, this specification is analogous to rough CIR in \cite{jaisson2016rough} or rough Heston in \cite{el2019characteristic}.

\subsection{Numerical implementation and option pricing}
Based on the convergence result, we will simulate the time series \eqref{qsdlong} to approximate the rough volatility model \eqref{eq: newroughvol}.

The baseline calculation fixes the parameters in Table \ref{pricing_result}, simulates the corresponding volatility paths with the D\&C FFT implementation, and uses those paths to price European and path dependent claims.

The parameter choices in this part are aligned with those adopted in prior numerical studies \citep{callegaro2021fast,ma2022fast,wang2025rough}:
\[\begin{split}
    &X_0=\log 100,\quad \Lambda_0=\log \sqrt{0.0392}, \quad \bar\rho=-0.681,\\
    &\kappa=0.1, \quad \theta=\log \sqrt{0.3156}, \quad v\mu=0.331, \quad \alpha=0.62.
\end{split}\]
Note that the value of $\Lambda_0$ and $\theta$ implies that volatility begins with $\sqrt{0.0392}\approx 19.80\%$, and the long term level is $\sqrt{0.3156}\approx 56.18\%$. The negative correlation $\bar\rho=-0.681$ captures the leverage channel. The parameter $\kappa$ controls the strength of mean reversion in the Volterra drift, $v\mu$ determines the log volatility noise scale in the limiting model through $v\kappa\mu$, and $\alpha$ determines the roughness index $H=\alpha-1/2$. Thus the baseline $\alpha=0.62$ corresponds to $H=0.12$.
To visualize the volatility behavior, we first simulate Equation \eqref{qsdlong} over a time horizon $T = 5$ with $n = 1000$, varying the roughness parameter $\alpha$. Figure \ref{fig:volatility} shows the path of the volatility $\sigma_t=\exp(\Lambda_t)$ by simulating \eqref{qsdlong} for different $\alpha$. 
\begin{figure}[htbp]
    \centering
    \includegraphics[width=1\textwidth]{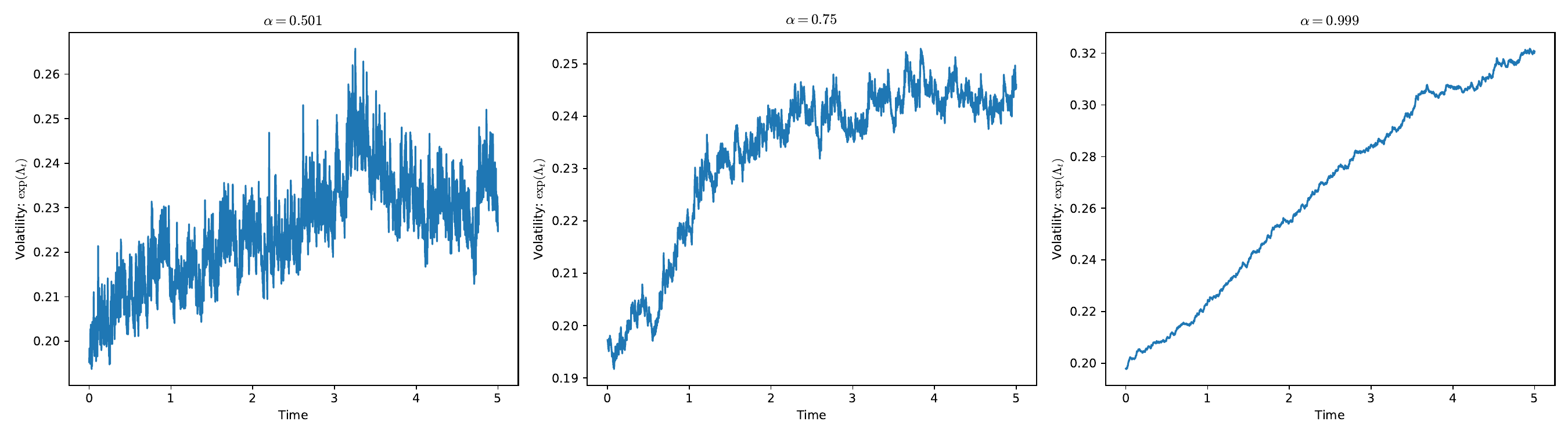}
    \caption{Volatility paths for different $\alpha$ values}
    \label{fig:volatility}
\end{figure}

It can be observed that as $\alpha$ increases, the volatility paths exhibit greater smoothness, and the mean-reverting behavior becomes more evident. In other words, a smaller $\alpha$ not only implies rougher paths but also corresponds to a slower mean-reversion, which is consistent with the observations reported in \cite{gatheral2018volatility}. 

Accordingly, the simulated paths can be employed to price a variety of options via Monte Carlo methods.

We proceed to price a range of path-independent and path-dependent options, including European, Asian, Lookback, and Barrier options, using Monte Carlo simulations. Algorithm \ref{keyalgorithm} outlines the computational steps for simulating asset paths and pricing these options.

Let $S^{(n)}=\exp(X^{(n)})$ denote the interpolated stock price generated by \eqref{qsdlong}, and let $S=\exp(X)$ denote the limiting stock price in Theorem \ref{newroughvol}. For each payoff type listed in the algorithm, write $\Phi$ for the corresponding functional on continuous paths and let $\Phi_n$ denote its discrete version on $\mathcal T_n=\{i/n:0\le i\le n\}$, with time averages, extrema, and barrier events evaluated at these dates when they occur. We write
\[
    P_n(\Phi)=\mb E[\Phi_n(S^{(n)})],
    \qquad
    P(\Phi)=\mb E[\Phi(S)] .
\]
The reported prices are discounted prices with $r=0$. The convergence argument has two steps. The extended continuous mapping theorem first transfers $S^{(n)}\Rightarrow S$ to $\Phi_n(S^{(n)})\Rightarrow \Phi(S)$ when $s_n\to s$ implies $\Phi_n(s_n)\to\Phi(s)$ at the continuity points of the limiting functional. This continuity is direct for the European, Asian, and lookback functionals; for the barrier functionals it follows from the no atom property at the barrier levels. Passing from payoff convergence in distribution to convergence of prices is then immediate for bounded payoffs, while the unbounded call-type payoffs require uniform integrability. For the score-driven model \eqref{qsdlong}, the condition $\mu\bar\rho\le0$ also yields uniform integrability of the discrete running maxima and hence convergence of lookback call prices. Specifically, we obtain the following proposition, whose proof is provided in Section S8 of the supplementary material.

\begin{proposition}\label{prop:option_price_convergence}
Consider the score-driven model \eqref{qsdlong} on $[0,1]$ with $\bar\rho\in(-1,1)$, $\mu\bar\rho\le0$, and $B_u>S_0>B_d>0$. Then, for every payoff $\Phi$ considered in Algorithm \ref{keyalgorithm},
\[
    P_n(\Phi)\longrightarrow P(\Phi).
\]
\end{proposition}

\begin{algorithm}
\caption{Monte Carlo Option Pricing with EGARCH($\infty$) Approximation}
\label{keyalgorithm}
\begin{algorithmic}
    \State \textbf{Input:} Model parameters ($\alpha$, $\kappa$, $\theta$, $\mu$, $\rho$), initial value ($X_0$, $\Lambda_0$), number of simulated paths $M$, time steps per year $n$, time horizon $T$, strike price $K$, up and in barrier $B_u$, down and out barrier $B_d$. 
    \State \textbf{Output:} Option prices and standard errors.
    \State \textbf{Initialization}: Compute discrete model parameters:  
    \[
    \phi_1 = 1 - \frac{1}{\Gamma(1-\alpha) \cdot 2^\alpha}, \quad \phi_i = \frac{1}{\Gamma(1-\alpha)} \left( \frac{1}{i^\alpha} - \frac{1}{(i+1)^\alpha} \right) \text{ for } i \geq 2.
    \]
    $\omega_n=\frac{\theta}{\mu\sqrt{n}}$, $a_n=1-\kappa n^{-\alpha}$, $\xi=\frac{\Lambda_0}{\theta}$. 
    \State \textbf{Key step}: Generate $M$ asset price paths $S_t=\ee^{X_t}$ by simulating Equation \eqref{qsdlong}. This step is accelerated by the D\&C FFT routine in Algorithm \ref{alg:dc_fft}. 
    \For{each path $S^{(i)}$}
        \State For each required payoff $\Phi_n$, compute $V_i=\Phi_n(S^{(i)})$:
            \begin{itemize}
                \item European: $\max(S_T - K, 0)$ for call, $\max(K - S_T, 0)$ for put
                \item Asian: $\max(\frac{1}{nT} \sum_{j=1}^{nT} S_{t_j}- K, 0)$ for call, $\max(K-\frac{1}{nT} \sum_{j=1}^{nT} S_{t_j}, 0)$ for put
                \item Lookback: $\max(\max_{j=0,\ldots,nT} S_{t_j} - K, 0)$ for call, $\max(K-\min_{j=0,\ldots,nT} S_{t_j}, 0)$ for put
                \item Barrier: $\max(S_T - K, 0)\mathbf 1_{\{\max_{j=0,\ldots,nT} S_{t_j}\geq B_u\}}$ for call, $\max(K - S_T, 0)\mathbf 1_{\{\min_{j=0,\ldots,nT} S_{t_j}> B_d\}}$ for put
            \end{itemize}
    \EndFor
    \State \textbf{return} \textit{option price estimates $V=\frac{1}{M}\sum_{i=1}^M V_i$, and standard errors $\sqrt{\frac{\sum_{i=1}^M (V_i-V)^2}{(M-1)M}}$.}
\end{algorithmic}
\end{algorithm}

Since the reported prices are obtained from simulated score-driven paths, we first describe the convolution based implementation and its numerical behavior.

Let $N = \lf nT\rf$ denote the total number of time steps. The simulation of \eqref{qsdlong} can be reformulated as a convolution operation:
\[
\sum_{i=1}^{t} \phi_i^{(n)} \left[ \lambda_{t-i}^{(n)} + \rho \varepsilon_{t-i} + \zeta \left( |\varepsilon_{t-i}| - \sqrt{2/\pi} \right) \right]=\left(\Phi^{(n)}*\Psi^{(n)}\right)_t,
\]
where $\Phi^{(n)} := (\phi_i^{(n)})_{i=1,\ldots,N}$ and $\Psi^{(n)} := \left( \lambda_{i}^{(n)} + \rho \varepsilon_{i} + \zeta \left( |\varepsilon_{i}| - \sqrt{2/\pi} \right) \right)_{i=0,\ldots,N-1}$. It inspires us to accelerate the computation using the Fast Fourier Transform (FFT). However, a standard global FFT cannot be applied directly to the entire sequence. This is due to the causal dependency inherent in the system: the term $\Psi_t^{(n)}$ is unknown until the current log volatility $\lambda_t^{(n)}$ has been evaluated. In a naive implementation, evaluating this sum step by step for $t=1, \dots, N$ imposes a computational complexity of $\mathcal{O}(N^2)$ per path. To this end, we adopt the D\&C FFT algorithm. The core idea is to decompose the causal convolution into a series of noncausal linear convolutions that can be efficiently evaluated in the frequency domain. 

Suppose we aim to evaluate the sequence over a time interval $[t_1, t_2)$. We partition this interval at the midpoint $t_{\text{mid}} =\lf\frac{t_1+t_2}{2}\rf$. For any time step $t$ in the right half ($t \in [t_{\text{mid}}, t_2)$), the convolution sum in Equation \eqref{qsdlong} can be decomposed as
\begin{equation}
\label{eq:split_conv}
\sum_{k=0}^{t-1} \phi_{t-k}^{(n)} \Psi_k^{(n)} = \underbrace{\sum_{k=t_1}^{t_{\text{mid}}-1} \phi_{t-k}^{(n)} \Psi_k^{(n)}}_{\text{Cross term}} + \sum_{k=t_{\text{mid}}}^{t-1} \phi_{t-k}^{(n)} \Psi_k^{(n)} + \underbrace{\sum_{k=0}^{t_1-1} \phi_{t-k}^{(n)} \Psi_k^{(n)}}_{\text{already accumulated}}.
\end{equation}
The algorithm processes this recursively: Once the left subinterval $[t_1, t_{\text{mid}})$ is fully evaluated, the sequence $(\Psi_{t}^{(n)})_{t\in[t_1, t_{\text{mid}}-1]}$ becomes entirely known. Consequently, the ``Cross term'' in Equation \eqref{eq:split_conv} transforms into a standard discrete linear convolution between a known signal block $\Psi^{(n)}$ and the fractional kernel $\Phi^{(n)}$. In this way, we can use the standard FFT algorithm for computation, with a complexity of $\mathcal{O}(N\log N)$. The results are then added in advance to the log volatility array $(\lambda_{t}^{(n)})_{t\in[t_\text{mid},t_2)}$ before the algorithm proceeds to solve the right subinterval. For details, see Algorithm \ref{alg:dc_fft}. Since the recurrence relation for the computational cost follows $T(N) = 2T(N/2) + \mathcal{O}(N \log N)$, the total theoretical complexity per path is $\mathcal{O}(N \log^2 N)$. 

The cross term is evaluated by zero-padded linear convolution to avoid circular convolution. In Algorithm \ref{alg:dc_fft}, write $\mathcal B_t=\rho\varepsilon_t+\zeta(|\varepsilon_t|-\sqrt{2/\pi})$ and let $\operatorname{LinConv}_{L}$ denote the linear convolution computed with FFT length $L$ and output indices starting at $0$. For $M$ simulated paths, the inputs form an $M\times N$ array. The CUDA implementation evaluates the convolutions as batched FFTs using NVIDIA's \texttt{cuFFT} library; when necessary, the path dimension is split into smaller batches to meet memory constraints. This batching affects throughput, but the complexity for each path remains $\mathcal{O}(N\log^2N)$.

This architecture shifts the problem from sequential latency to parallel throughput, significantly reducing the execution time for large-scale simulations involving one million paths or more.

\begin{algorithm}[htbp]
\caption{D\&C FFT algorithm}
\label{alg:dc_fft}

\begin{algorithmic}[1]
\Require Time interval $[t_1, t_2)$, precomputed kernel $\Phi^{(n)}$, and quasi score innovation array $\mathcal{B}$
\Statex \textbf{Initialize:} $\lambda^{(n)}\gets(\omega_t^{(n)})_{t=0}^{N-1}$ and $\Psi^{(n)}\gets(0)_{t=0}^{N-1}$
\Procedure{Solve\_DC}{$t_1, t_2$}
    \If{$t_2 - t_1 = 1$} \Comment{Base case}
        \State $t \gets t_1$
        \State $\Psi^{(n)}_t \gets \lambda_t^{(n)} + \mathcal{B}_t$
    \Else \Comment{Recursive step}
        \State $t_{\text{mid}} \gets \lf\frac{t_1+t_2}{2}\rf$
        \State \Call{Solve\_DC}{$t_1, t_{\text{mid}}$} \Comment{left half}
        
        \Statex \Comment{Cross term by FFT}
        \State $\Psi^{(n)}_{\mathrm{left}} \gets (\Psi^{(n)}_k)_{k=t_1}^{t_{\text{mid}}-1}$
        \State $\Phi^{(n)}_{\mathrm{slice}} \gets (\phi_k^{(n)})_{k=1}^{t_2-t_1-1}$
        \State Choose $L_{\mathrm{fft}}\ge |\Psi^{(n)}_{\mathrm{left}}|+|\Phi^{(n)}_{\mathrm{slice}}|-1$
        \State $C \gets \operatorname{LinConv}_{L_{\mathrm{fft}}}(\Psi^{(n)}_{\mathrm{left}},\Phi^{(n)}_{\mathrm{slice}})$
        \State $\lambda_t^{(n)} \gets \lambda_t^{(n)}+C_{t-t_1-1}$ for $t=t_{\text{mid}},\ldots,t_2-1$
        
        \State \Call{Solve\_DC}{$t_{\text{mid}}, t_2$} \Comment{right half}
    \EndIf
\EndProcedure
\end{algorithmic}

\end{algorithm}

Table \ref{tab:converge_performance} reports the convergence and computational performance of the GPU-accelerated D\&C FFT framework under two complementary experiments. Simulations were conducted on a single NVIDIA GeForce RTX 5060 Ti GPU with PyTorch 2.9 and CUDA 12.9. Option prices correspond to a European ATM call ($K=S_0=100$, $r=0$, $T=1$) from the score-driven approximation to \eqref{eq: newroughvol} with $\alpha=0.62$. In Panel A, the high resolution reference run gives $p^\ast=8.4441$ at $n^\ast=5000$ and $M=10^6$ paths, with standard error $0.0138$. The reported absolute errors measure $|\widehat P_{n,M}-\widehat P_{n^\ast,M}|$ and include Monte Carlo noise from both estimates.

Panel A illustrates the effect of time discretization by varying $n$ from 50 to 4000 while keeping $M=10^6$ fixed. The simulated prices move toward the high resolution reference as the grid is refined, with the absolute error falling from $0.1790$ at $n=50$ to $0.0005$ at $n=4000$. The small reversal at $n=1000$ relative to $n=500$ is within the Monte Carlo noise of this comparison, since the reported error compares two simulation estimates. The runtime increases with $n$ in a way broadly consistent with the $\mathcal{O}(N\log^2N)$ per path complexity.

Panel B investigates Monte Carlo sampling error by increasing $M$ from $10^4$ to $10^6$ while keeping $n=2000$ fixed. The standard error decreases from $0.1330$ to $0.0138$, in line with the usual $\mathcal{O}(M^{-1/2})$ Monte Carlo rate. For fixed $n$, the computational work is expected to scale approximately linearly in $M$ once the GPU is sufficiently utilized. The full range fit in Figure \ref{fig:convergence} has a slope around $0.76$ on logarithmic axes because the smallest path volumes include fixed GPU overhead and device under-utilization; for the larger runs, the timings are close to linear: doubling $M$ from $250{,}000$ to $500{,}000$ and then to $1{,}000{,}000$ approximately doubles the elapsed time at each step.

\begin{table}[htbp]
    \centering
    \caption{Numerical convergence and simulation time of the GPU-accelerated D\&C FFT implementation}
    \label{tab:converge_performance}
    \begin{threeparttable}
    
    \begin{tabular*}{0.9\textwidth}{@{\extracolsep{\fill}} l c c c c c }
        \toprule
        \multicolumn{6}{l}{\textbf{Panel A: Discretization Convergence}\ ($M = 1,000,000$ fixed)} \\
        \multicolumn{6}{l}{High resolution reference run ($n=5,000$): 8.4441 (0.0138)} \\
        \midrule
        $n$ & Price & $|$Error$|$ & Rel. Err (\%) & Std. Err & Time (s) \\
        \midrule
        50    & 8.2653 & 0.1790 & 2.1170 & 0.0135 & 0.171 \\
        100   & 8.3484 & 0.0957 & 1.1331 & 0.0136 & 0.403 \\
        200   & 8.3865 & 0.0576 & 0.6821 & 0.0137 & 0.957 \\
        500   & 8.4298 & 0.0143 & 0.1690 & 0.0137 & 2.555 \\
        1,000 & 8.4157 & 0.0284 & 0.3363 & 0.0137 & 5.344 \\
        2,000 & 8.4467 & 0.0026 & 0.0307 & 0.0138 & 13.978 \\
        4,000 & 8.4446 & 0.0005 & 0.0062 & 0.0138 & 64.592 \\
        \bottomrule
    \end{tabular*}

    \vspace{0.1cm}

    \begin{tabular*}{0.9\textwidth}{@{\extracolsep{\fill}} l c c c }
        \multicolumn{4}{l}{\textbf{Panel B: Monte Carlo Sampling and Runtime}\ ($n = 2,000$ fixed)} \\
        \midrule
        $M$ & Price & Std. Err & Time (s) \\
        \midrule
        10,000    & 8.1703 & 0.1330 & 0.519 \\
        25,000    & 8.3146 & 0.0856 & 0.569 \\
        50,000    & 8.3516 & 0.0608 & 0.679 \\
        100,000   & 8.4194 & 0.0432 & 1.648 \\
        250,000   & 8.4359 & 0.0274 & 3.426 \\
        500,000   & 8.4273 & 0.0194 & 6.864 \\
        1,000,000 & 8.4467 & 0.0138 & 13.403 \\
        \bottomrule
    \end{tabular*}
    \begin{tablenotes}
    \footnotesize
    \item Notes: In Panel A, $|\mathrm{Error}|$ is defined as $|\widehat P_{n,M}-\widehat P_{n^\ast,M}|$, where $\widehat P_{n^\ast,M}$ is the high resolution reference run computed at $n^\ast=5000$ with the same number of paths. Std. Err. denotes the Monte Carlo standard error of the corresponding price estimate. All simulations use the GPU-accelerated D\&C FFT algorithm on a single NVIDIA GeForce RTX 5060 Ti GPU (16 GB VRAM).
    \end{tablenotes}
    \end{threeparttable}
\end{table}

\begin{figure}[htbp]
    \centering
    \includegraphics[width=0.95\textwidth]{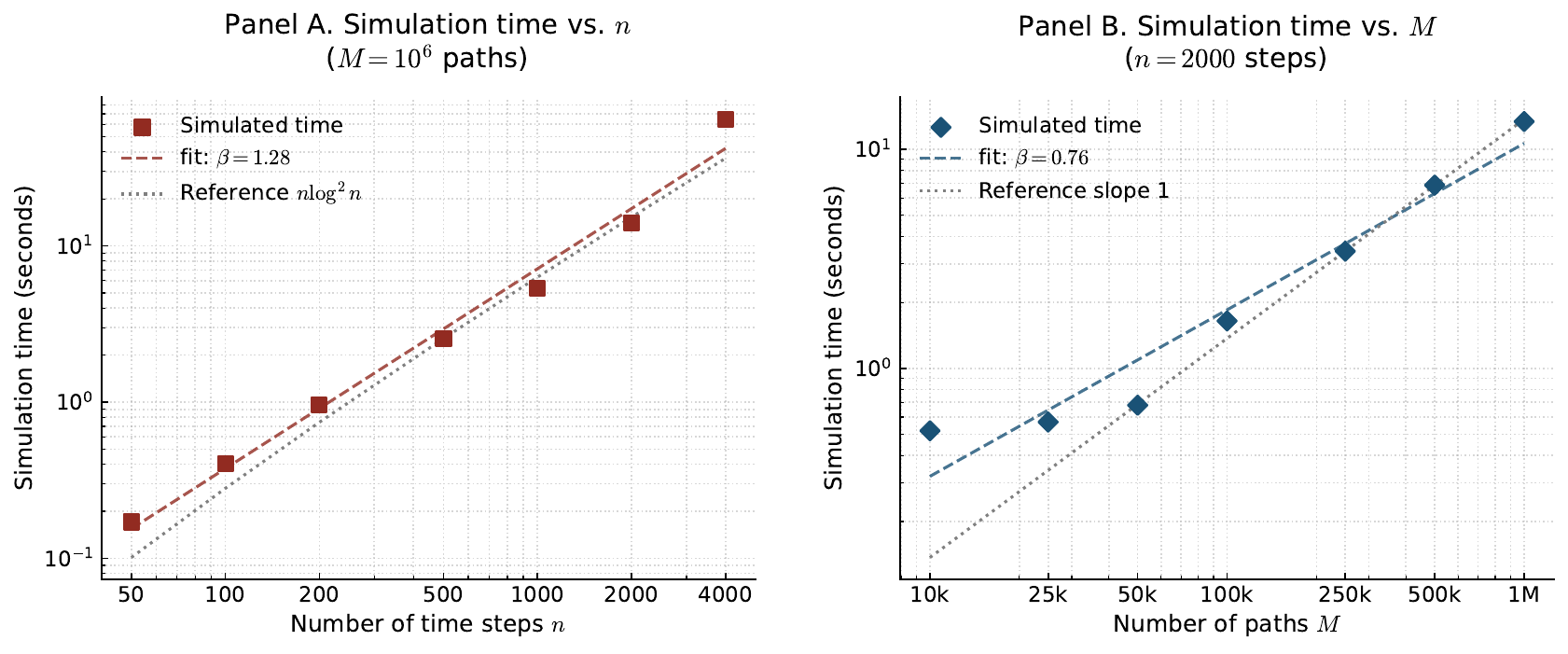}
    \caption{Simulation time versus time steps $n$ and path volume $M$ for the GPU-accelerated D\&C FFT implementation. Fitted slopes are reported on logarithmic axes; the path-volume fit includes small-$M$ overhead, while the larger-$M$ timings are close to linear.}
    \label{fig:convergence}
\end{figure}

Table \ref{pricing_result} reports the option prices computed via Algorithm \ref{keyalgorithm} for different values of the roughness parameter $\alpha$, with standard errors shown in parentheses. We set $M=1\times 10^6$, $n = 1000$, $T = 1$, with strike prices $K = 80, 90, 100, 110, 120$. For the barrier options, the barrier level is set to $B_u = 110$ for the Up-In call and $B_d = 90$ for the Down-Out put. In these simulations, as $\alpha$ decreases, European calls and puts prices rise, reflecting greater sensitivity to extreme price movements under rougher volatility paths. Asian options follow similar patterns but with attenuated sensitivity, as price averaging dampens the impact of path roughness. Lookback options show the most pronounced positive sensitivity to roughness, as their payoffs directly depend on path extrema. Barrier options exhibit the most intricate response. While barrier calls increase with roughness, corresponding barrier puts decrease in value. This occurs because roughness affects both terminal payoff variability and barrier triggering probabilities. 

\begin{singlespace}
\begin{table}[htbp]
\centering
\caption{Option pricing results from the score-driven approximation to \eqref{eq: newroughvol} for different $\alpha$}
\label{pricing_result}
\begin{threeparttable}
\begin{tabular}{c cccccccc}
\toprule
%=============== PANEL A ===============
\multicolumn{9}{l}{\textbf{Panel A: $\alpha = 0.999$}} \\
\midrule
\multirow{2}{*}{Strike} & \multicolumn{2}{c}{European} & \multicolumn{2}{c}{Asian} & \multicolumn{2}{c}{Lookback} & \multicolumn{2}{c}{Barrier} \\
\cmidrule(r){2-3} \cmidrule(r){4-5} \cmidrule(r){6-7} \cmidrule(r){8-9}
& Call & Put & Call & Put & Call & Put & Call & Put \\
\midrule
80  & 21.3879 & 1.3879  & 20.1144 & 0.1066  & 37.1352 & 2.5754  & 18.5064 & 0.0000  \\
    & \scriptsize(0.0190) & \scriptsize(0.0041) & \scriptsize(0.0115) & \scriptsize(0.0008) & \scriptsize(0.0148) & \scriptsize(0.0053) & \scriptsize(0.0208) & \scriptsize(0.0000) \\
90  & 13.8855 & 3.8855  & 11.1055 & 1.0977  & 27.1352 & 7.2071  & 12.9496 & 0.0000  \\
    & \scriptsize(0.0167) & \scriptsize(0.0073) & \scriptsize(0.0103) & \scriptsize(0.0029) & \scriptsize(0.0148) & \scriptsize(0.0087) & \scriptsize(0.0171) & \scriptsize(0.0000) \\
100 & 8.2796  & 8.2796  & 4.6697  & 4.6619  & 17.1352 & 15.3069 & 8.1547  & 0.1730  \\
    & \scriptsize(0.0135) & \scriptsize(0.0109) & \scriptsize(0.0074) & \scriptsize(0.0063) & \scriptsize(0.0148) & \scriptsize(0.0106) & \scriptsize(0.0136) & \scriptsize(0.0009) \\
110 & 4.5487  & 14.5488 & 1.4525  & 11.4448 & 9.2785  & 25.3069 & 4.5487  & 1.0970  \\
    & \scriptsize(0.0103) & \scriptsize(0.0142) & \scriptsize(0.0042) & \scriptsize(0.0093) & \scriptsize(0.0130) & \scriptsize(0.0106) & \scriptsize(0.0103) & \scriptsize(0.0033) \\
120 & 2.3208  & 22.3208 & 0.3413  & 20.3336 & 4.6746  & 35.3069 & 2.3208  & 2.9594  \\
    & \scriptsize(0.0074) & \scriptsize(0.0168) & \scriptsize(0.0020) & \scriptsize(0.0110) & \scriptsize(0.0099) & \scriptsize(0.0106) & \scriptsize(0.0074) & \scriptsize(0.0065) \\
\midrule[0.85pt]
%=============== PANEL B ===============
\multicolumn{9}{l}{\textbf{Panel B: $\alpha = 0.62$}} \\
\midrule
\multirow{2}{*}{Strike} & \multicolumn{2}{c}{European} & \multicolumn{2}{c}{Asian} & \multicolumn{2}{c}{Lookback} & \multicolumn{2}{c}{Barrier} \\
\cmidrule(r){2-3} \cmidrule(r){4-5} \cmidrule(r){6-7} \cmidrule(r){8-9}
& Call & Put & Call & Put & Call & Put & Call & Put \\
\midrule
80  & 21.4807 & 1.4808  & 20.1336 & 0.1258  & 37.3842 & 2.7573  & 18.6839 & 0.0000  \\
    & \scriptsize(0.0193) & \scriptsize(0.0043) & \scriptsize(0.0117) & \scriptsize(0.0009) & \scriptsize(0.0150) & \scriptsize(0.0056) & \scriptsize(0.0210) & \scriptsize(0.0000) \\
90  & 14.0207 & 4.0209  & 11.1721 & 1.1642  & 27.3842 & 7.4788  & 13.1092 & 0.0000  \\
    & \scriptsize(0.0169) & \scriptsize(0.0075) & \scriptsize(0.0104) & \scriptsize(0.0031) & \scriptsize(0.0150) & \scriptsize(0.0089) & \scriptsize(0.0173) & \scriptsize(0.0000) \\
100 & 8.4245  & 8.4247  & 4.7532  & 4.7453  & 17.3842 & 15.6111 & 8.3016  & 0.1614  \\
    & \scriptsize(0.0137) & \scriptsize(0.0111) & \scriptsize(0.0075) & \scriptsize(0.0065) & \scriptsize(0.0150) & \scriptsize(0.0108) & \scriptsize(0.0138) & \scriptsize(0.0009) \\
110 & 4.6717  & 14.6719 & 1.5022  & 11.4943 & 9.4974  & 25.6111 & 4.6717  & 1.0352  \\
    & \scriptsize(0.0105) & \scriptsize(0.0144) & \scriptsize(0.0043) & \scriptsize(0.0095) & \scriptsize(0.0132) & \scriptsize(0.0108) & \scriptsize(0.0105) & \scriptsize(0.0032) \\
120 & 2.4080  & 22.4082 & 0.3584  & 20.3505 & 4.8325  & 35.6111 & 2.4080  & 2.8183  \\
    & \scriptsize(0.0076) & \scriptsize(0.0170) & \scriptsize(0.0020) & \scriptsize(0.0111) & \scriptsize(0.0101) & \scriptsize(0.0108) & \scriptsize(0.0076) & \scriptsize(0.0064) \\
\midrule[0.85pt]
%=============== PANEL C ===============
\multicolumn{9}{l}{\textbf{Panel C: $\alpha = 0.501$}} \\
\midrule
\multirow{2}{*}{Strike} & \multicolumn{2}{c}{European} & \multicolumn{2}{c}{Asian} & \multicolumn{2}{c}{Lookback} & \multicolumn{2}{c}{Barrier} \\
\cmidrule(r){2-3} \cmidrule(r){4-5} \cmidrule(r){6-7} \cmidrule(r){8-9}
& Call & Put & Call & Put & Call & Put & Call & Put \\
\midrule
80  & 21.4954 & 1.4955  & 20.1395 & 0.1315  & 37.4156 & 2.7886  & 18.7047 & 0.0000  \\
    & \scriptsize(0.0193) & \scriptsize(0.0043) & \scriptsize(0.0118) & \scriptsize(0.0009) & \scriptsize(0.0150) & \scriptsize(0.0056) & \scriptsize(0.0210) & \scriptsize(0.0000) \\
90  & 14.0418 & 4.0419  & 11.1908 & 1.1827  & 27.4156 & 7.5252  & 13.1300 & 0.0000  \\
    & \scriptsize(0.0169) & \scriptsize(0.0076) & \scriptsize(0.0104) & \scriptsize(0.0031) & \scriptsize(0.0150) & \scriptsize(0.0090) & \scriptsize(0.0173) & \scriptsize(0.0000) \\
100 & 8.4466  & 8.4468  & 4.7760  & 4.7680  & 17.4156 & 15.6625 & 8.3230  & 0.1590  \\
    & \scriptsize(0.0138) & \scriptsize(0.0111) & \scriptsize(0.0075) & \scriptsize(0.0065) & \scriptsize(0.0150) & \scriptsize(0.0109) & \scriptsize(0.0138) & \scriptsize(0.0009) \\
110 & 4.6902  & 14.6904 & 1.5156  & 11.5075 & 9.5248  & 25.6624 & 4.6902  & 1.0238  \\
    & \scriptsize(0.0105) & \scriptsize(0.0144) & \scriptsize(0.0043) & \scriptsize(0.0095) & \scriptsize(0.0132) & \scriptsize(0.0109) & \scriptsize(0.0105) & \scriptsize(0.0032) \\
120 & 2.4209  & 22.4211 & 0.3629  & 20.3549 & 4.8522  & 35.6624 & 2.4209  & 2.7937  \\
    & \scriptsize(0.0076) & \scriptsize(0.0170) & \scriptsize(0.0020) & \scriptsize(0.0112) & \scriptsize(0.0101) & \scriptsize(0.0109) & \scriptsize(0.0076) & \scriptsize(0.0063) \\
\bottomrule
\end{tabular}
\end{threeparttable}
\end{table}
\end{singlespace}

For the baseline rough case $\alpha=0.62$, Table \ref{tab:mawu_benchmark} compares the score-driven prices with two independent Monte Carlo discretizations of the corresponding continuous Volterra equation. The fast Monte Carlo (Fast MC) scheme follows the sum of exponentials kernel approximation of \citet{ma2022fast}, whereas the Euler--Maruyama (EM) simulation directly discretizes the same equation. All three runs use $M=10^6$, $n=1000$, $T=1$, and the same five strikes as in Table \ref{pricing_result}. The score-driven D\&C FFT prices differ from the Fast MC and EM simulation columns by at most $0.0158$ and $0.0157$, respectively, which is comparable to the reported Monte Carlo errors. In terms of runtime, to ensure comparability, we also use GPU acceleration for the other two benchmarks. The total wall clock times for computing the option prices at the five strikes in Table \ref{tab:mawu_benchmark} are $6.65$, $7.32$, and $80.90$ seconds for the score-driven D\&C FFT, Fast MC, and EM simulation runs, respectively. It can be seen that, under the same GPU-accelerated setting, our method is slightly faster than Fast MC and much faster than the EM simulation.

\begin{table}[htbp]
\centering
\caption{European call comparison with Volterra Monte Carlo discretizations}
\label{tab:mawu_benchmark}
\begin{threeparttable}
\begin{tabular}{@{}c c c c c c@{}}
\toprule
Strike & Score-driven D\&C FFT & Fast MC & Abs. diff. & EM simulation & Abs. diff. \\
\midrule
80  & 21.5154 & 21.5008 & 0.0147 & 21.4998 & 0.0157 \\
    & \scriptsize(0.0193) & \scriptsize(0.0193) & & \scriptsize(0.0193) & \\
90  & 14.0578 & 14.0420 & 0.0158 & 14.0432 & 0.0147 \\
    & \scriptsize(0.0169) & \scriptsize(0.0169) & & \scriptsize(0.0169) & \\
100 & 8.4578  & 8.4474  & 0.0104 & 8.4468  & 0.0110 \\
    & \scriptsize(0.0138) & \scriptsize(0.0138) & & \scriptsize(0.0138) & \\
110 & 4.6967  & 4.6946  & 0.0021 & 4.6892  & 0.0075 \\
    & \scriptsize(0.0105) & \scriptsize(0.0105) & & \scriptsize(0.0105) & \\
120 & 2.4262  & 2.4253  & 0.0009 & 2.4182  & 0.0080 \\
    & \scriptsize(0.0076) & \scriptsize(0.0076) & & \scriptsize(0.0076) & \\
\bottomrule
\end{tabular}
\begin{tablenotes}
\footnotesize
\item Notes: Standard errors are reported in parentheses under the corresponding prices. Abs. diff. is the absolute difference from the score-driven D\&C FFT price. All columns use $M=10^6$, $n=1000$, and $T=1$ with the same model parameters as in Table \ref{pricing_result}. The Fast MC column uses a CUDA implementation of the sum of exponentials approximation with $58$ exponential terms.
\end{tablenotes}
\end{threeparttable}
\end{table}

The case $\alpha=1$ provides a smooth limit comparison. In this case, the rough kernel becomes nonsingular and \eqref{eq: newroughvol} reduces to the classical lognormal stochastic volatility model of \citet{scott1987option}:
\[\begin{split}
    &\dd X_t=-\frac{1}{2} \ee^{2\Lambda_t}\dd t+\ee^{\Lambda_t} \dd B_t,\\
    &\dd \Lambda_t=\kappa(\theta-\Lambda_t)\dd t+\sigma\dd W_t,\\
    &\text{Cov}(B_t, W_t)=\bar\rho t,
\end{split}\]
where $\sigma=v\kappa\mu$. Using the same parameter values and Monte Carlo settings as in the previous experiments, we report the option prices in Table \ref{pricing_result_alpha1}. The numbers in parentheses represent the percentage errors relative to the rough volatility model with $\alpha=0.999$. The errors are mostly within 0.5\%, with minor deviations attributable to the diminishing roughness. These results confirm that the option prices computed from our model are consistent with those from the classical specification.

\begin{singlespace}
\begin{table}[htbp]
\centering
\caption{Option pricing results under the lognormal stochastic volatility model (case $\alpha=1$)}
\label{pricing_result_alpha1}
\begin{threeparttable}
\begin{tabular}{c cccccccc}
\toprule
\multirow{2}{*}{Strike} & \multicolumn{2}{c}{European} & \multicolumn{2}{c}{Asian} & \multicolumn{2}{c}{Lookback} & \multicolumn{2}{c}{Barrier} \\
\cmidrule(r){2-3} \cmidrule(r){4-5} \cmidrule(r){6-7} \cmidrule(r){8-9}
& Call & Put & Call & Put & Call & Put & Call & Put \\
\midrule
80  & 21.3982 & 1.3971 & 20.1032 & 0.1088 & 37.1375 & 2.5846 & 18.5158 & 0.0000 \\ 
    & \scriptsize(0.05\%) & \scriptsize(0.66\%) & \scriptsize(0.06\%) & \scriptsize(2.02\%) & \scriptsize(0.01\%) & \scriptsize(0.36\%) & \scriptsize(0.05\%) & \scriptsize(0.00\%) \\ 
90  & 13.8995 & 3.8984 & 11.1007 & 1.1064 & 27.1375 & 7.2211 & 12.9619 & 0.0000 \\ 
    & \scriptsize(0.10\%) & \scriptsize(0.33\%) & \scriptsize(0.04\%) & \scriptsize(0.79\%) & \scriptsize(0.01\%) & \scriptsize(0.19\%) & \scriptsize(0.09\%) & \scriptsize(0.00\%) \\ 
100 & 8.2904  & 8.2893  & 4.6688  & 4.6745  & 17.1375 & 15.3238 & 8.1651  & 0.1716 \\ 
    & \scriptsize(0.13\%) & \scriptsize(0.12\%) & \scriptsize(0.02\%) & \scriptsize(0.27\%) & \scriptsize(0.01\%) & \scriptsize(0.11\%) & \scriptsize(0.13\%) & \scriptsize(0.82\%) \\ 
110 & 4.5568  & 14.5557 & 1.4506  & 11.4563 & 9.2866  & 25.3238 & 4.5568  & 1.0898 \\ 
    & \scriptsize(0.18\%) & \scriptsize(0.05\%) & \scriptsize(0.13\%) & \scriptsize(0.10\%) & \scriptsize(0.09\%) & \scriptsize(0.07\%) & \scriptsize(0.18\%) & \scriptsize(0.66\%) \\ 
120 & 2.3246  & 22.3235 & 0.3369  & 20.3425 & 4.6842  & 35.3238 & 2.3246  & 2.9450 \\ 
    & \scriptsize(0.16\%) & \scriptsize(0.01\%) & \scriptsize(1.31\%) & \scriptsize(0.04\%) & \scriptsize(0.20\%) & \scriptsize(0.05\%) & \scriptsize(0.16\%) & \scriptsize(0.49\%) \\ 
\bottomrule
\end{tabular}
\begin{tablenotes}
\footnotesize
\item Notes: Percentages in parentheses are relative differences from the corresponding nearly smooth case $\alpha=0.999$.
\end{tablenotes}
\end{threeparttable}
\end{table}
\end{singlespace}

We finally examine the model implied at the money (ATM) volatility skew term structure. Let

\[\psi(\tau):=\left|\frac{\partial \sigma_{\text{IV}}(k,\tau)}{\partial k}\right|_{k=0},\]
where $k=\log(K/F_\tau)$ denotes forward log moneyness and $\tau$ represents the time to maturity. In the simulations, $r=0$ and hence $F_\tau=S_0$.

Empirical evidence suggests that the volatility skew exhibits a power-law decay as time to maturity increases \citep{gatheral2018volatility}. Theoretical work by \citet{fukasawa2011asymptotic} established that when volatility is driven by fractional Brownian motion with Hurst parameter $H$, the ATM skew asymptotically follows $\psi(\tau) \sim \tau^{H - 1/2}$ for small $\tau$. For our limit \eqref{eq: newroughvol}, the roughness index is $H=\alpha-1/2$, so the predicted power-law decay is $\psi(\tau)\sim \tau^{\alpha-1}$. We compute the ATM volatility skew generated from our model when $\alpha\in\{0.6,0.7,0.8, 0.999\}$, using $M=10^6$, $n=1000$, strikes $K=70,75,\ldots,130$, and maturities $0.15,0.175,\ldots,1.00$, as shown in Figure \ref{fig:skew}. To assess this power-law relationship, we regress $\log|\psi(\tau)|$ on $\log\tau$, so the regression coefficient is the slope on logarithmic axes. The fitted slopes for $\alpha=0.600,0.700,0.800$ are $-0.402$, $-0.310$, and $-0.212$, close to the predicted values $-0.400$, $-0.300$, and $-0.200$, with corresponding $R^2$ values $0.961$, $0.928$, and $0.835$. For the nearly smooth case $\alpha=0.999$, the fitted slope is $-0.005$, reflecting the almost flat skew in Figure \ref{fig:skew}. These results are consistent with prior studies: rougher volatility paths generate steeper initial skews that subsequently decay more rapidly as $\tau$ increases, while in the nearly smooth case the volatility skew is observed to approximate a constant across maturities.

\begin{figure}[htbp]
    \centering
    \includegraphics[width=0.8\textwidth]{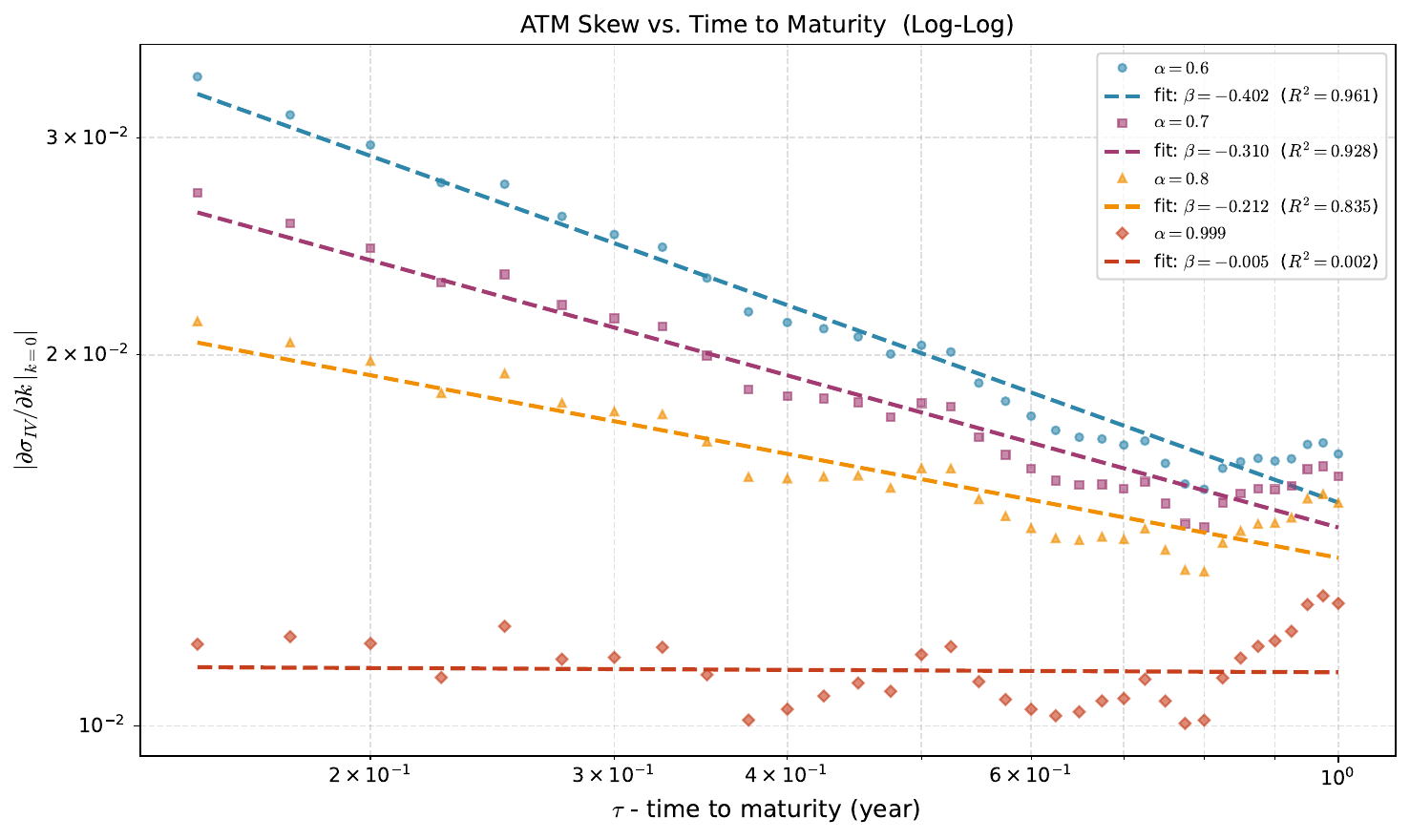}
    \caption{ATM implied volatility skew on logarithmic axes. For each maturity $\tau$, the ATM skew is estimated by local cubic regression of implied volatility on log moneyness \citep{fengler2005semiparametric}.}
    \label{fig:skew}
\end{figure}

As an empirical illustration, we fit our model to out-of-the-money (OTM) SPX option quotes over six consecutive trading days from July 24 to July 31, 2023, using $n=1000$; forward prices are constructed from put-call parity with the FRED three month Treasury rate. Table \ref{tab:spx_atm_skew_calibration} reports the calibrated parameters and implied volatility fit errors. The market skew exponent is computed as $\alpha_{\rm skew}^{\rm Mkt}=1+\beta_{\rm Mkt}$, where $\beta_{\rm Mkt}$ is the slope from regressing $\log|\psi_{\rm Mkt}(\tau)|$ on $\log\tau$ over maturities up to 60 days. The calibrated roughness parameter remains in the rough range across these dates and is of the same order as the FIGAS$(0,d,1)$ estimates around $d=0.65$ reported by \citet{opschoor2019fractional} for realized covariance kernels of S\&P 500 equities. The estimated correlation is negative throughout, and the implied persistence $\hat a_n=1-\hat\kappa n^{-\hat\alpha}$ stays close to one. Figure \ref{fig:spx_calibration_smiles_20230728} illustrates the fitted smiles for July 28 at four maturities. At these maturities, the fitted curves track the market level and slope around the ATM region.

\begin{table}[htbp]
\centering
\caption{SPX option calibration over consecutive trading days}
\label{tab:spx_atm_skew_calibration}
\begin{threeparttable}
\begingroup
\setlength{\tabcolsep}{3.0pt}
\begin{tabular}{lccccccccc}
\toprule
Date & $\hat\alpha$ & Mkt. $\alpha_{\rm skew}$ & $\hat H$ & $\hat{\bar\rho}$ & $\hat a_n$ & $\hat\kappa$ & $\hat\mu$ & IV RMSE & ATM IV RMSE \\
\midrule
July 24 & 0.768 & 0.713 & 0.268 & $-0.671$ & 0.989 & 2.157 & 0.679 & 0.0058 & 0.0053 \\
July 25 & 0.766 & 0.666 & 0.266 & $-0.688$ & 0.992 & 1.516 & 0.887 & 0.0057 & 0.0039 \\
July 26 & 0.801 & 0.719 & 0.301 & $-0.681$ & 0.991 & 2.161 & 0.746 & 0.0045 & 0.0054 \\
July 27 & 0.857 & 0.682 & 0.357 & $-0.707$ & 0.993 & 2.781 & 0.709 & 0.0048 & 0.0039 \\
July 28 & 0.798 & 0.758 & 0.298 & $-0.676$ & 0.990 & 2.400 & 0.673 & 0.0045 & 0.0048 \\
July 31 & 0.809 & 0.685 & 0.309 & $-0.707$ & 0.989 & 2.970 & 0.580 & 0.0046 & 0.0059 \\
\bottomrule
\end{tabular}
\begin{tablenotes}
\footnotesize
\item Notes: Each daily calculation uses between 4,351 and 4,883 OTM SPX option quotes. ATM IV RMSE is computed for quotes in the ATM moneyness bucket, roughly $|k|\le 0.05$. The reported model prices use $M=10^6$ Monte Carlo paths. The IV RMSE values are reported in implied volatility units.
\end{tablenotes}
\endgroup
\end{threeparttable}
\end{table}

\begin{figure}[htbp]
    \centering
    \includegraphics[width=0.82\textwidth]{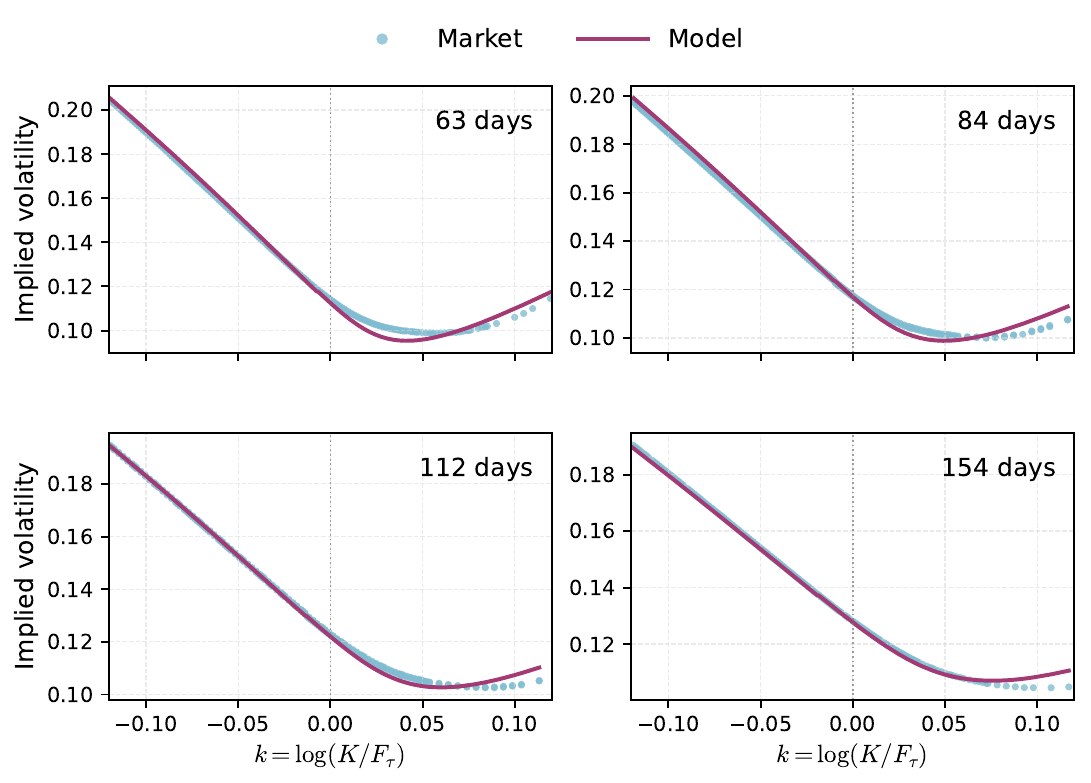}
    \caption{SPX implied volatility smiles on July 28, 2023. Dots are market implied volatilities and solid lines are fitted model implied volatilities for four maturities. The horizontal axis is forward log moneyness $k=\log(K/F_\tau)$ near the money.}
    \label{fig:spx_calibration_smiles_20230728}
\end{figure}

\section{Conclusion}\label{sec:conclusion}
This study reveals the connection between long memory score-driven models and the rough OU process, the former offering statistical robustness and the latter exhibiting rich dynamic behavior, especially in volatility. Specifically, we find that when $\alpha\in(1/2,1)$, the coefficients of the lag terms decay at a power-law rate of order $1+\alpha$, and the limiting process is driven by a fractional Brownian motion with Hurst parameter $H=\alpha-1/2$, which implies roughness. This result is inspired by the work of \cite{jaisson2016rough} on scaling limits of Hawkes processes. In our proof, however, we introduce a novel approach based on GFOs. Using their continuity in H\"older spaces, we obtain H\"older convergence for the auxiliary continuous representation and then transfer the convergence to the original parameter process in the Skorokhod space through asymptotic indistinguishability.  

We apply our theoretical results to volatility modeling. When the time varying parameter governed by the score-driven mechanism represents the log volatility, the model converges to a new rough volatility model in which the log volatility follows a rough OU process. More precisely, it is a process obtained by applying a Riemann-Liouville fractional operator of order $\alpha$ to an OU process, which differs from that in \cite{gatheral2018volatility}. 

In the numerical experiments, we approximate this rough volatility model using an EGARCH($\infty$) specification. As a quasi score-driven model,  it incorporates the score of asymmetric distributions and can thus capture the leverage effect. In this sense, our results may be seen as a long memory extension of \cite{nelson1990arch}'s result for AR$(1)$ Exponential ARCH. The proposed Monte Carlo algorithm, optimized through D\&C FFT convolution and parallel GPU computing, achieves significant computational efficiency, making it suitable for pricing not only European options but also various path-dependent options, including Asian, Lookback, and Barrier options. The option pricing results show the expected convergence and the ATM skew regressions match the predicted scaling; in the SPX calibration, the resulting roughness and leverage estimates are also empirically plausible.

The results of this paper can be further developed in future work. In fact, the derivation of the OU-type limiting process relies on the assumption that the conditional score scale $U$ is asymptotically constant. This condition is satisfied, for instance, when the scaling function $S$ is chosen as the $-1/2$ power of the Fisher information. Extending the result to more general score scales $U$, where state dependent diffusion coefficients may arise in the Volterra limit, would require new techniques, as the GFO approach used here would break down: the limit process in this case cannot be represented as a GFO applied to the associated diffusion.

\section*{Acknowledgements}
The authors express gratitude to Yingli Wang for insightful discussions and valuable suggestions, which significantly improved the quality of this work.

%% The Appendices part is started with the command \appendix;
%% appendix sections are then done as normal sections
\appendix

\section{Proof of Proposition \ref{weak_converge}}\label{proof:weak_converge}
The process under consideration is defined as
\[Y_t^{(n)} = \theta t + \frac{\theta}{\sqrt{n}\omega_n} S_n(t),\]
where $S_n(t)$ represents the normalized sum of martingale differences:
\[S_n(t) = \frac{1}{\sqrt{n}} \left( \sum_{i=0}^{\lf nt \rf - 1} \Delta M^{(n)}_i + (nt - \lf nt \rf) \Delta M^{(n)}_{\lf nt \rf} \right).\]
We establish the convergence of $S_n$ in the H\"{o}lder space $\mathcal{C}^{\eta}([0,1])$ to the Brownian motion $\gamma W$ by verifying the two standard components of a martingale functional limit theorem: finite dimensional convergence and tightness. Before proving the martingale central limit theorem, we first show that the unscaled parameter process escapes every fixed compact set at positive times. This delocalization result uses the nondegeneracy in Assumption \ref{a1} and the deterministic renewal kernel estimates proved above. It is then used to justify the predictable quadratic variation convergence.

\begin{lemma}[Escape from compact sets]\label{lem:escape}
    For every $s\in(0,1]$ and every $M>0$,
    \[
        \mb P\left(|\lambda_{\lfloor ns\rfloor}^{(n)}|\le M\right)
        \to 0 .
    \]
\end{lemma}

\begin{proof}
See Section S7 of the supplementary material. 
\end{proof}

For later use we record the following consequence of Lemma \ref{lem:escape}. For $T\in(0,1]$ and $M>0$, put
\[
    A_n^M(T):=
    \frac1n\sum_{k=0}^{\lfloor nT\rfloor-1}
    \mathbf 1_{\{|\lambda_k^{(n)}|\le M\}} .
\]
Then
\[
\begin{split}
    \mb E A_n^M(T)
    &=
    \frac1n\sum_{k=0}^{\lfloor nT\rfloor-1}
    \mb P(|\lambda_k^{(n)}|\le M)  \\
    &=
    \int_0^T
    \mb P\left(|\lambda_{\lfloor ns\rfloor}^{(n)}|\le M\right)\dd s
    +O(n^{-1}).
\end{split}
\]
The integrand converges to zero for every fixed $s>0$ by Lemma \ref{lem:escape}, and it is bounded by one. Dominated convergence and Markov's inequality therefore give $A_n^M(T)\toP0$. Assumption \ref{a1} implies that $U$ is globally bounded, so $C_U:=\sup_{x\in\mb R}|U^2(x)-\gamma^2|<\infty$. Given $\varepsilon>0$, choose $M$ such that $\sup_{|x|>M}|U^2(x)-\gamma^2|<\varepsilon$. Then, for all $t\in[0,T]$,
\[
\begin{split}
    &\left|
    \frac1n\sum_{k=0}^{\lfloor nt\rfloor-1}
    U^2(\lambda_k^{(n)})-\gamma^2t
    \right| \\
    &\qquad\le
    \varepsilon T
    +
    C_U A_n^M(T)
    +
    \frac{\gamma^2}{n}.
\end{split}
\]
Since $\varepsilon$ is arbitrary,
\[
    \sup_{0\le t\le T}
    \left|
    \frac1n\sum_{k=0}^{\lfloor nt\rfloor-1}
    U^2(\lambda_k^{(n)})-\gamma^2t
    \right|
    \toP0 .
\]

\medskip
\noindent
\textit{Step 1: Convergence of finite dimensional distributions.}

We first establish the weak convergence $S_n \Rightarrow \gamma W$ in the Skorokhod space $D([0,1])$, which implies the convergence of finite dimensional distributions. To this end, we verify the conditions of classical FCLT for martingales \citep{hall2014martingale}. By the preceding compact occupation argument, the predictable quadratic variation of the martingale part satisfies
\[
    V_n(t)
    :=
    \frac1n\sum_{k=0}^{\lfloor nt\rfloor-1}
    \mb E\!\left[
        (\Delta M_k^{(n)})^2
        \mid\mathcal F_{k-1}^{(n)}
    \right]
    =
    \frac1n\sum_{k=0}^{\lfloor nt\rfloor-1}
    U^2(\lambda_k^{(n)})
    \to \gamma^2t
\]
uniformly on compact time intervals in probability. This is the predictable quadratic variation condition required by the martingale functional central limit theorem. Throughout this proof, for $k=0$ the $\sigma$-field $\mathcal F_{-1}^{(n)}$ denotes the initial $\sigma$-field.
    When this condition is satisfied, the Lindeberg condition is equivalent to proving that 
    \[\max_{0\leq k\leq \lf nt\rf-1}\frac{1}{\sqrt{n}}|\Delta M_k^{(n)}|\toP 0.\]
    For any $\varepsilon>0$, applying the Markov inequality and Assumption \ref{a3} yields
\[\begin{split}
    \mb{P}(\max_{0\leq k\leq \lf nt\rf-1} \frac{1}{\sqrt{n}}|\Delta M_k^{(n)}|>\varepsilon)\leq \frac{1}{n^2\varepsilon^4}\sum_{k=0}^{\lf nt\rf-1}\mb{E}|\Delta M_k^{(n)}|^4\leq C_{\varepsilon}\frac{\lf nt\rf}{n^2}\to0. 
\end{split}\]
The validity of these conditions implies that $S_n \Rightarrow \gamma W$ in $D([0,1])$. Consequently, the finite dimensional distributions of $S_n$ converge to those of $\gamma W$.

\medskip
\noindent
\textit{Step 2: Tightness in $\mathcal{C}_0^{\eta}([0,1])$.}

For every $\eta\in(0,1/2)$, the piecewise linear paths of $S_n$ belong to the separable little H\"{o}lder space $\mathcal C_0^\eta([0,1])$, and the sample paths of $\gamma W$ belong to this space almost surely. We therefore verify tightness in $\mathcal C_0^\eta([0,1])$. The identity map from $\mathcal C_0^\eta([0,1])$ into $\mathcal C^\eta([0,1])$ is continuous \citep{horvath2024functional}, so this gives the convergence stated in Proposition \ref{weak_converge}.
According to Theorem 6 in \cite{ravckauskas2007holderian}, tightness in $\mathcal C_0^\eta([0,1])$ requires two conditions: (i) tightness of the sequence of random variables $\{S_n(t)\}$ for each dyadic $t \in [0,1]$, and (ii) the uniform decay of the dyadic projections. The convergence of finite dimensional distributions established in Step 1 ensures that the first condition holds. For $j\ge1$, let $\mathcal K_j:=\{1,3,\ldots,2^j-1\}$ denote the indices of the new dyadic points at level $j$. Therefore, it suffices to verify the second condition regarding the Schauder coefficients: for any $\varepsilon > 0$,
\begin{equation}\label{tightness_criterion}
    \lim_{J \to \infty} \limsup_{n \to \infty} \mathbb{P}\Big( \sup_{j \geq J} 2^{\eta j} \max_{k\in\mathcal K_j} |\delta_{j,k}(S_n)| > \varepsilon \Big) = 0,
\end{equation}
where $\delta_{j,k}(S_n)$ denotes the Schauder coefficients of $S_n$, defined by
\[
\begin{split}
        \delta_{j,k}(S_n) &= S_n(k/2^j) - \frac{S_n((k-1)/2^j) + S_n((k+1)/2^j)}{2} \\
        &= -\frac{1}{2} \left( \Delta S_n(k/2^j) - \Delta S_n((k-1)/2^j) \right),
    \end{split}
\]
with $\Delta S_n(t) := S_n(t+2^{-j}) - S_n(t)$. Since $S_n$ is the linear interpolation of the normalized martingale partial sums, for $j\le \lfloor\log_2 n\rfloor$ the coefficient $\delta_{j,k}(S_n)$ can be written as a martingale transform
\[
    \delta_{j,k}(S_n)
    =
    n^{-1/2}\sum_{i\in I_{n,j,k}}
    c_{i,n,j,k}\Delta M_i^{(n)},
\]
where the coefficients $c_{i,n,j,k}$ are deterministic and uniformly bounded. We take
\[
    I_{n,j,k}
    =
    [\lfloor (k-1)n/2^j\rfloor,\lfloor (k+1)n/2^j\rfloor]\cap\mb Z,
\]
so that $|I_{n,j,k}|\le Cn2^{-j}$ for $j\le \lfloor\log_2 n\rfloor$.
    When $n<2^j$, sharper bounds apply: if the interval $[(k-1)n/2^j, (k+1)n/2^j]$ contains no integer points, $\delta_{j,k}(S_n)=0$; otherwise,
    \begin{equation}\label{sharp_bound}
        \left|\delta_{j,k}(S_n)\right| \leq \frac{n^{1/2}}{2^{j-1}}\left (\left|\Delta M_{\lf kn/2^j\rf}^{(n)}\right|\vee \left|\Delta M_{\lf (k-1)n/2^j\rf}^{(n)}\right|\right )=: \Delta_{n,j,k}^*
    \end{equation}
    and at most $2n$ such intervals contain integer points. 

We split the probability estimate in \eqref{tightness_criterion} at the critical scale $j \approx \log_2 n$. For $p > 2$,
\begin{equation}\label{tight_prob}
    \begin{split}
        &\mb{P}\left( \sup_{j \geq J} 2^{\eta j} \max_{k\in\mathcal K_j} |\delta_{j,k}(S_n)| > \varepsilon \right) \\
        &\leq \sum_{j \geq J} \sum_{k\in\mathcal K_j} \mb{P}\left( |\delta_{j,k}(S_n)| > \frac{\varepsilon}{2^{\eta j}} \right) \\
        &\leq \sum_{J \leq j \leq \lf \log_2 n \rf} \sum_{k\in\mathcal K_j} \frac{\mb{E}|\delta_{j,k}(S_n)|^p}{(\varepsilon 2^{-\eta j})^p} + 2n \sum_{j > \lf \log_2 n \rf} \frac{\mb{E}|\Delta_{n,j,k}^*|^p}{(\varepsilon 2^{-\eta j})^p}.
    \end{split}
\end{equation}
Thus the key estimate relies on controlling the $p$ moments of $\delta_{j,k}(S_n)$ at the coarse dyadic scales and of $\Delta_{n,j,k}^*$ at the fine dyadic scales. For the first part, Rosenthal's inequality for martingale transforms gives, for some constant $C_p$,
    \[\begin{split}
        \mb{E}\left|\delta_{j,k}(S_n)\right|^p
        &\leq C_p\left[\mb{E}\left(n^{-1}\sum_{i\in I_{n,j,k}}c_{i,n,j,k}^2U^2(\lambda_i^{(n)})\right)^{p/2}
        +n^{-p/2}\sum_{i\in I_{n,j,k}}\mb{E}\left|c_{i,n,j,k}\Delta M_i^{(n)}\right|^p\right]\\
        &\leq C\left[\left(n^{-1}|I_{n,j,k}|\right)^{p/2}+n^{-p/2}|I_{n,j,k}|\right]\quad (\text{by Assumptions \ref{a1} and \ref{a3}})\\
        &\leq C'\left(n^{-1} |I_{n,j,k}|\right)^{p/2}.
    \end{split}
    \]
Since $j\leq \lf\log_2 n\rf$, the bound on $|I_{n,j,k}|$ gives
    \[\mb{E}\left|\delta_{j,k}(S_n)\right|^p\leq C 2^{-jp/2}.\]
Consequently, the first summation in \eqref{tight_prob} satisfies
    \[\sum_{J\leq j\leq \lf\log_2 n\rf}\sum_{k\in\mathcal K_j}\frac{\mb{E}\left|\delta_{j,k}(S_n)\right|^p}{(\varepsilon 2^{-\eta j})^p}\leq \frac{C}{\varepsilon^p}\sum_{J\leq j\leq \lf\log_2 n\rf} 2^{(1+\eta p-p/2)j}.\]
For the second part, using \eqref{sharp_bound}, we have
    \[2n \sum_{j> \lf\log_2 n\rf}\frac{\mb{E}|\Delta_{n,j,k}^*|^p}{(\varepsilon 2^{-\eta j})^p}\leq \frac{4C_p}{\varepsilon^p}n^{1+p/2} \sum_{j> \lf\log_2 n\rf}2^{(\eta -1)pj}.\]
Since $\eta<1$, the geometric series $\sum_{j> \lf\log_2 n\rf}2^{(\eta -1)pj}$ converges and is of order $\mathcal{O}(n^{(\eta-1)p})$. Therefore, the second summation is $\mathcal{O}(n^{1+\eta p-p/2})$, which vanishes as $n\to\infty$, provided $\eta<\frac{p-2}{2p}$. 

Combining these estimates yields
    \[\limsup_{n\to\infty}\mb{P}\left(\sup_{j\geq J} 2^{\eta j}\max_{k\in\mathcal K_j}\left|\delta_{j,k}(S_n)\right|>\varepsilon\right)\leq \frac{2^p C}{\varepsilon^p}\frac{2^{(1+\eta p-p/2)J}}{1-2^{1+\eta p-p/2}}\to 0, \quad \text{as} \ J\to\infty.\]
Since $p>2$ can be taken arbitrarily large, tightness holds for all $\eta\in(0,1/2)$.

\clearpage
\begingroup
\setcounter{section}{0}
\setcounter{equation}{0}
\setcounter{figure}{0}
\setcounter{table}{0}
\setcounter{theorem}{0}
\renewcommand{\thesection}{S\arabic{section}}
\renewcommand{\theequation}{S\arabic{equation}}
\renewcommand{\thefigure}{S\arabic{figure}}
\renewcommand{\thetable}{S\arabic{table}}
\renewcommand{\thetheorem}{S\arabic{theorem}}
\renewcommand{\thelemma}{S\arabic{theorem}}
\renewcommand{\theproposition}{S\arabic{theorem}}
\renewcommand{\thecorollary}{S\arabic{theorem}}

\begin{center}
{\Large\bfseries Supplementary Material for:\\[0.4em]
``Long memory score-driven models as approximations for rough Ornstein-Uhlenbeck processes''}

\vspace{0.8em}
Yinhao Wu \qquad Ping He
\end{center}

\vspace{1em}
\noindent This document serves as the Supplementary Material for the manuscript titled
``\textit{Long memory score-driven models as approximations for rough Ornstein-Uhlenbeck processes}''.
It contains detailed mathematical proofs that were omitted from the main manuscript due to space constraints.

\vspace{1em}
\hrule
\vspace{2em}

\section{Proof of Proposition 2}

We first introduce the notation and preliminaries used throughout the proof. When the kernel is chosen as $g(x)=\frac{x^\alpha}{\Gamma(\alpha+1)}$ in (3.2), the operator $\mathcal{G}^\alpha$ becomes the Riemann-Liouville fractional integral operator of order $\alpha$. For $\alpha\in(1/2,1)$, we denote $I^\alpha=\mathcal{G}^\alpha$ and $D^\alpha=\mathcal{G}^{-\alpha}$. When $\alpha=1$, $I^1$ and $D^1$ represent the standard integral and differential operators, respectively. Furthermore, we define $F^{\alpha,\kappa}(t) := \int_0^t f^{\alpha,\kappa}(s) \dd s$, the cumulative distribution function of the Mittag-Leffler density $f^{\alpha,\kappa}(t)$. Then, the following standard but important identity holds:
\begin{equation}\label{identity}
    I^{1-\alpha} f^{\alpha,\kappa}(t) = \kappa (1 - F^{\alpha,\kappa}(t)), 
\end{equation}
which can be readily verified by examining the Laplace transforms of both sides.

As Equation (3.1) is actually a linear stochastic Volterra equation, the existence and uniqueness of its strong solution are guaranteed by \cite{coutin2001stochastic}. It therefore suffices to show that $\Lambda_t$ satisfies this equation. We apply the fractional integral operator $I^{1-\alpha}$ to $\Lambda_t$. Let $K_t = I^{1-\alpha} \Lambda_t$. We decompose $K_t$ into its deterministic and stochastic components,
\[K_t = K_t^{\text{det}} + K_t^{\text{stoch}}=I^{1-\alpha}\left(\theta F^{\alpha,\kappa}(t) \right)+I^{1-\alpha}\left(\mu\gamma\int_0^t f^{\alpha,\kappa}(t-s)\dd W_s\right).\]

For the deterministic part, by using identity \eqref{identity} we obtain
\[\begin{split}
    K_t^{\text{det}} &= I^{1-\alpha} [\theta F^{\alpha,\kappa}](t) = \theta \cdot I^{1-\alpha} [I^1 f^{\alpha,\kappa}](t)=\theta \cdot I^1 [I^{1-\alpha} f^{\alpha,\kappa}](t)\\
    &=\theta \int_0^t \left( I^{1-\alpha} f^{\alpha,\kappa}(s) \right) \dd s = \kappa\theta \int_0^t (1 - F^{\alpha,\kappa}(s)) \dd s=\kappa\theta t-\kappa\theta\int_0^t F^{\alpha,\kappa}(s) \dd s.
\end{split}\]
Similarly, combining with the stochastic Fubini theorem yields the stochastic part as follows: 
\[\begin{split}
    K_t^{\text{stoch}} &= \kappa\mu\gamma \int_0^t (1 - F^{\alpha,\kappa}(t-s)) \dd W_s=\kappa\mu\gamma \left[W_t - \int_0^t \int_0^u f^{\alpha,\kappa}(u-s) \dd W_s \dd u\right]\\
    &=\kappa\mu\gamma \left[W_t - \int_0^t \frac{1}{\mu\gamma}\left(\Lambda_u-\theta F^{\alpha,\kappa}(u)\right)\dd u\right]=\kappa\mu\gamma W_t-\kappa\int_0^t\Lambda_s\dd s+\kappa\theta\int_0^t F^{\alpha,\kappa}(s) \dd s\\
\end{split}\]
Thus, we have
\[K_t=\kappa\theta t+\kappa\mu\gamma W_t-\kappa\int_0^t\Lambda_s\dd s= \kappa \int_0^t (\theta - \Lambda_s) \dd s + \kappa\mu\gamma W_t\]

Now, we apply the fractional derivative operator $D^{1-\alpha}$ to $K_t$ to recover $\Lambda_t$. According to the property $D^{1-\alpha} I^1 = I^\alpha$ and the result on the action of $\mathcal{G}^{\alpha-1}$ on diffusion processes established in Proposition 5, we obtain
\[\begin{split}
    \Lambda_t&=\kappa D^{1-\alpha}I^1(\theta - \Lambda)(t)+\kappa\mu\gamma D^{1-\alpha}W_t=\kappa I^{\alpha}(\theta - \Lambda)(t)+\kappa\mu\gamma D^{1-\alpha}W_t\\
    &=\frac{1}{\Gamma(\alpha)} \int_0^t (t-s)^{\alpha-1} \kappa(\theta - \Lambda_s) \dd s+\frac{\gamma\kappa\mu}{\Gamma(\alpha)}\int_0^t (t-s)^{\alpha-1}\dd W_s. \\
\end{split}\]
This is precisely Equation (3.1), which completes the proof.

\section{Proof of Lemma 3}

Throughout this proof all Fourier transforms are taken under the convention
\[
    \widehat g(u)=\int_{\mb{R}}\ee^{-\ii ux}g(x)\dd x.
\]
We first record a lower bound needed for the argument. Since $\phi_k>0$ for all $k\ge 1$ and $\sum_{k\ge 1}\phi_k=1$, the distribution of the lag variable is aperiodic on the integer lattice. Hence $|\widehat\phi(v)|<1$ for $v\notin 2\pi\mb Z$. The Karamata-Tauberian expansion gives
\[
    1-\widehat\phi(v)
    \sim
    \frac{K\Gamma(1-\alpha)}{\alpha}(\ii v)^\alpha ,
    \qquad v\to0.
\]
Taking real parts yields constants $\eta_0\in(0,\pi)$ and $c_0>0$ such that
\[
    \operatorname{Re}\{1-\widehat\phi(v)\}\ge c_0|v|^\alpha,
    \qquad |v|\le \eta_0.
\]
Combining this estimate near $2\pi\mb Z$ with compactness away from $2\pi\mb Z$, we obtain, for all sufficiently large $n$,
\begin{equation}\label{denom_lower_supp}
    |1-a_n\widehat\phi(v)|
    \ge
    c\{(1-a_n)+d(v,2\pi\mb Z)^\alpha\},
    \qquad v\in\mb R .
\end{equation}

Choose $\delta\in(0,\eta_0)$. From the definition of the histogram density in the main text,
\[
    \widehat \rho_n(u)
    =
    \ee^{\ii u/n}
    \frac{1-\ee^{-\ii u/n}}{\ii u/n}
    \frac{(1-a_n)\widehat\phi(u/n)}
         {1-a_n\widehat\phi(u/n)} .
\]
For $|u|\le \delta n$, using $|m_n(u)|\le1$ and \eqref{denom_lower_supp},
\[
    |\widehat\rho_n(u)|
    \le
    \frac{C(1-a_n)}{(1-a_n)+|u/n|^\alpha}
    \le
    \frac{C}{1+|u|^\alpha},
\]
where we used Assumption 2. This proves the low frequency bound.

For the high frequency tail, set $v=u/n$. Then
\[
    \int_{|u|>\delta n}|\widehat\rho_n(u)|^2\dd u
    =
    n\int_{|v|>\delta}
    \left|\frac{1-\ee^{-\ii v}}{\ii v}\right|^2
    \left|
    \frac{(1-a_n)\widehat\phi(v)}
         {1-a_n\widehat\phi(v)}
    \right|^2\dd v .
\]
Split $\mb R$ into intervals $I_\ell=[2\pi\ell-\pi,2\pi\ell+\pi]$ and write $v=2\pi\ell+r$. Since $|1-\ee^{-\ii v}|\le C|r|$ and the denominator satisfies \eqref{denom_lower_supp},
\[
\begin{split}
    &\int_{|v|>\delta}
    \left|\frac{1-\ee^{-\ii v}}{\ii v}\right|^2
    \left|
    \frac{(1-a_n)\widehat\phi(v)}
         {1-a_n\widehat\phi(v)}
    \right|^2\dd v  \\
    &\qquad\le
    C(1-a_n)^2
    \left[
    1+
    \sum_{\ell\ne0}\frac{1}{1+\ell^2}
    \int_{-\pi}^{\pi}
    \frac{r^2}{\{(1-a_n)+|r|^\alpha\}^2}\dd r
    \right].
\end{split}
\]
The integral in $r$ is uniformly bounded because $r^2/\{(1-a_n)+|r|^\alpha\}^2\le |r|^{2-2\alpha}$ and $\alpha<1$. Hence
\[
    \int_{|u|>\delta n}|\widehat\rho_n(u)|^2\dd u
    \le Cn(1-a_n)^2\to0,
\]
as $\alpha>1/2$. Finally, for fixed $u$, the Karamata expansion and Assumption 2 imply
\[
    \frac{(1-a_n)\widehat\phi(u/n)}
         {1-a_n\widehat\phi(u/n)}
    \to
    \frac{\kappa}{\kappa+(\ii u)^\alpha},
\]
and $\ee^{\ii u/n}m_n(u)\to1$, proving the pointwise convergence.

\section{Proof of Proposition 4}

Extend both $\rho_n$ and $f^{\alpha,\kappa}$ by zero to the negative half line. Both functions belong to $L^2(\mb R)$: indeed, $\|\rho_n\|_2^2=n\sum_{k\ge1}p_{n,k}^2\le n$, while $f^{\alpha,\kappa}(x)\le Cx^{\alpha-1}$ near zero and $\alpha>1/2$. By Plancherel's identity,
\[
    \|\rho_n-f^{\alpha,\kappa}\|_{L^2(\mb R)}^2
    =
    \frac{1}{2\pi}
    \int_{\mb R}
    |\widehat\rho_n(u)-\widehat f^{\alpha,\kappa}(u)|^2\dd u .
\]
The limiting Fourier transform is
\[
    \widehat f^{\alpha,\kappa}(u)
    =
    \frac{\kappa}{\kappa+(\ii u)^\alpha},
    \qquad
    |\widehat f^{\alpha,\kappa}(u)|
    \le
    \frac{C}{1+|u|^\alpha}.
\]
Fix $R>0$. On $[-R,R]$, pointwise convergence and boundedness by one imply convergence of the integral by dominated convergence. For $R<|u|\le\delta n$, Lemma 3 gives
\[
    \int_{R<|u|\le \delta n}
    |\widehat\rho_n(u)-\widehat f^{\alpha,\kappa}(u)|^2\dd u
    \le
    C\int_{|u|>R}\frac{\dd u}{(1+|u|^\alpha)^2},
\]
which tends to zero as $R\to\infty$. For $|u|>\delta n$,
\[
    \int_{|u|>\delta n}
    |\widehat\rho_n(u)-\widehat f^{\alpha,\kappa}(u)|^2\dd u
    \le
    2\int_{|u|>\delta n}|\widehat\rho_n(u)|^2\dd u
    +
    2\int_{|u|>\delta n}|\widehat f^{\alpha,\kappa}(u)|^2\dd u
    \to0 .
\]
Combining the three regions and then letting $R\to\infty$ proves convergence in $L^2(\mb R)$, hence also on $[0,1]$.

\section{Proof of Proposition 9}

Write $t_m=m/n$, $m=0,\ldots,n$. Since $\Lambda^{(n)}$ is constant on each interval $[t_m,t_{m+1})$ and $\widetilde\Lambda^{(n)}\Rightarrow\Lambda$ in $C([0,1])$, tightness of $\widetilde\Lambda^{(n)}$ in $C([0,1])$ gives
\[
    \sup_{\substack{s,t\in[0,1]\\ |s-t|\le 1/n}}
    |\widetilde\Lambda_s^{(n)}-\widetilde\Lambda_t^{(n)}|
    \toP 0 .
\]
It is therefore enough to prove
\begin{equation}\label{grid_indistinguishable_supp}
    \max_{0\le m\le n}
    |\Lambda_{t_m}^{(n)}-\widetilde\Lambda_{t_m}^{(n)}|
    \toP0.
\end{equation}
At a grid point,
\[
\begin{split}
    \Lambda_{t_m}^{(n)}-\widetilde\Lambda_{t_m}^{(n)}
    &=
    (1-a_n)\theta
    +
    \theta\left\{
    \sum_{k=1}^{m}(1-a_n)\psi_k^{(n)}
    -
    \int_0^{t_m}f^{\alpha,\kappa}(s)\dd s
    \right\} \\
    &\quad+
    \frac{\theta}{n\omega_n}
    \sum_{i=0}^{m-1}
    e_{n,m-i}\Delta M_i^{(n)},
\end{split}
\]
where
\[
    R_{n,k}:=n(1-a_n)\psi_k^{(n)},\qquad
    \bar f_{n,k}:=
    n\int_{(k-1)/n}^{k/n}f^{\alpha,\kappa}(s)\dd s,
    \qquad
    e_{n,k}:=R_{n,k}-\bar f_{n,k}.
\]
The deterministic terms converge uniformly to zero by the uniform convergence of the renewal distribution functions. It remains to control
\[
    Z_{n,m}:=\frac{\theta}{n\omega_n}
    \sum_{i=0}^{m-1}e_{n,m-i}\Delta M_i^{(n)} .
\]

We use two elementary kernel estimates. First,
\[
    \frac1n\sum_{k=1}^{n}e_{n,k}^2\to0.
\]
Indeed, $R_n(x):=\sum_{k\ge1}R_{n,k}\mathbf1_{((k-1)/n,k/n]}(x)=a_n\rho_n(x)$ converges to $f^{\alpha,\kappa}$ in $L^2(0,1)$ by Proposition 4, and the cell average approximation $\bar f_n$ also converges to $f^{\alpha,\kappa}$ in $L^2(0,1)$. Second, for all $1\le d\le n$,
\begin{equation}\label{kernel_translation_supp}
    \frac1n\sum_{k\ge1}(e_{n,k+d}-e_{n,k})^2
    +
    \frac1n\sum_{k=1}^{d}e_{n,k}^2
    \le
    C\left(\frac dn\right)^{2\alpha-1}.
\end{equation}
For $R_{n,k}$ this follows from discrete Parseval and the denominator lower bound in Lemma 3:
\[
    \frac1n\sum_{k\ge1}(R_{n,k+d}-R_{n,k})^2
    +
    \frac1n\sum_{k=1}^{d}R_{n,k}^2
    \le
    C\int_0^\infty
    \frac{1\wedge (dx/n)^2}{(1+x^\alpha)^2}\dd x
    \le C(d/n)^{2\alpha-1}.
\]
For $\bar f_{n,k}$, Plancherel's identity and
$|\widehat f^{\alpha,\kappa}(u)|\le C(1+|u|^\alpha)^{-1}$ give
\[
    \|f^{\alpha,\kappa}(\cdot+d/n)-f^{\alpha,\kappa}(\cdot)\|_2^2
    \le C(d/n)^{2\alpha-1},
\]
and Jensen's inequality transfers this bound to the cell averages; the boundary term is controlled by $f^{\alpha,\kappa}(x)\le Cx^{\alpha-1}$ near zero. This proves \eqref{kernel_translation_supp}.

Set $C_n=\theta/(n\omega_n)=O(n^{-1/2})$. For fixed $m$, the sum defining $Z_{n,m}$ is a martingale transform. Rosenthal's inequality for martingale transforms gives, for any $p>2$,
\[
\begin{split}
    \mb E|Z_{n,m}|^p
\le C_p\left[
    \left(
        C_n^2\sum_{k=1}^{m}e_{n,k}^2
    \right)^{p/2}
    +
    C_n^p\sum_{k=1}^{m}|e_{n,k}|^p
\right]
\le
C_p
\left(
\frac1n\sum_{k=1}^{m}e_{n,k}^2
\right)^{p/2},
\end{split}
\]
where Assumptions 1 and 3 control the conditional square-function and $p$-th moment terms, and the last inequality uses $C_n=O(n^{-1/2})$ and $\sum |e_{n,k}|^p\le(\sum e_{n,k}^2)^{p/2}$.
The right hand side tends to zero uniformly in $m\le n$. For increments, if $0\le\ell<m\le n$ and $d=m-\ell$, the same martingale transform argument gives
\[
    \mb E|Z_{n,m}-Z_{n,\ell}|^p
    \le
    C
    \left(\frac{|m-\ell|}{n}\right)^{p(\alpha-1/2)},
\]
because the square sum of the new weights is bounded by
\[
    \sum_{k=1}^{\ell}(e_{n,k+d}-e_{n,k})^2+\sum_{k=1}^{d}e_{n,k}^2,
\]
which is controlled by \eqref{kernel_translation_supp}; the $p$-th moment term is absorbed as above. Choose $p$ so large that $p(\alpha-1/2)>1$. Linear interpolation of the grid values then satisfies Kolmogorov's tightness criterion in $C([0,1])$, while the finite dimensional distributions converge to zero. Hence
\[
    \max_{0\le m\le n}|Z_{n,m}|\toP0,
\]
which proves \eqref{grid_indistinguishable_supp}. Combining the grid estimate with the modulus of continuity reduction completes the proof.

\section{Proof of Theorem 10}\label{proof:example_theorem}
The proof necessitates establishing the joint weak convergence of the return process and the volatility process. We first verify the assumptions of Theorem 1 to obtain the convergence of $\Lambda^{(n)}$. Denote $\mb{E}_{k-1}[\cdot]:=\mb{E}[\cdot|\mathcal{F}_{k-1}]$. By the properties of GED, for $p\geq 2$, the moments of the innovation term $\varepsilon_k$ satisfy:
\[\mb{E}_{k-1}[|\varepsilon_k|^p]=\frac{2^{p/\nu}}{\nu\Gamma(1+1/\nu)}\Gamma\left(\frac{p+1}{\nu}\right)<\infty.\]
Recall that the score function is defined as $\nabla_k=\frac{\nu}{2}|\varepsilon_k|^\nu-1$. We have, 
\[\mb{E}[|\nabla_{k}|^p]\leq 2^{p-1}\left[\left(\frac{\nu}{2}\right)^p\mb{E}|\varepsilon_k|^{\nu p}+1\right]<\infty. \]
In particular, the second moment is given by
\[U^2(\lambda_k)=\mb{E}_{k-1}[\nabla_{k}^2]=\frac{\nu^2}{4}\mb{E}_{k-1}[|\varepsilon|^{2\nu}]-\nu\mb{E}_{k-1}[|\varepsilon|^{\nu}]+1=\nu.\]
The dynamics of $\Lambda_t$ then follow from Theorem 1. 
    
Then we establish the joint convergence of driving noises. Consider the vector of partial sums representing the driving noises for the returns and the volatility parameter:
\[\mathbf{M}^{(n)}_t = \left(B^{(n)}_t, W^{(n)}_t\right) := \left(\sum_{i=1}^{\lf nt\rf} \frac{\varepsilon_i}{\sqrt{\chi n}}, \sum_{i=1}^{\lf nt\rf} \frac{\nabla_i}{\sqrt{\nu n}}\right).\]
Since the sequence $\{\varepsilon_i\}$ is i.i.d., we apply the multidimensional Donsker's invariance principle. Note that $\mb{E}_{k-1}(\varepsilon_k)=0$, and the score properties imply $\mb{E}_{k-1}(\nabla_k)=0$ and $\mb{E}_{k-1}(\varepsilon_k \nabla_k)=0$. Consequently, $\mathbf{M}^{(n)}$ converges weakly in the Skorokhod space $\mathcal{D}([0,1], \mb{R}^2)$ to a two dimensional Brownian motion $(B, W)$ with an identity covariance matrix. The joint partial sum convergence also holds for $B^{(n)}$ and the linearly interpolated score partial sum used in Theorem 1. By the continuity of the GFO map, applying this map to the second coordinate gives $(B^{(n)},\widetilde\Lambda^{(n)})\Rightarrow(B,\Lambda)$, and Proposition 9 gives $\|\widetilde\Lambda^{(n)}-\Lambda^{(n)}\|_\infty\toP0$. Hence
\begin{equation}\label{lambda_B}
    (\Lambda^{(n)}, B^{(n)}) \Rightarrow (\Lambda, B) \quad \text{in } \mathcal{D}([0,1], \mb{R}^2). 
\end{equation}

Finally, we address the convergence of $X^{(n)}$. Put $c_n=(1-a_n)\theta/\omega_n$ and define the predictable step process
\[
    \widehat H_s^{(n)}=\ee^{c_n\lambda_i^{(n)}},
    \qquad s\in((i-1)/n,i/n],\quad i=1,\ldots,n.
\]
Since $\lambda_i^{(n)}$ is $\mathcal F_{i-1}$ measurable, this process is predictable, and the discrete sum in the definition of $X^{(n)}$ satisfies
\[
    X_t^{(n)}=\int_0^t\widehat H_s^{(n)}\dd B_s^{(n)}.
\]
The process $\widehat H^{(n)}$ is the one-step time shift of $\ee^{\Lambda^{(n)}}$. Since \eqref{lambda_B} has a continuous limit, this shift is asymptotically negligible, and hence
\[
    (\widehat H^{(n)},B^{(n)},\Lambda^{(n)})
    \Rightarrow
    (\ee^\Lambda,B,\Lambda).
\]

We now invoke the convergence theorem for stochastic integrals (see, e.g., \cite{kurtz1991weak}). The conditions required are: (i) The joint convergence of the integrand and integrator, which is given above; and (ii) The uniform tightness of integrators, which is equivalent to verifying the quadratic variation is stochastically bounded. 
    
We observe that
\[\mb{E}([B^{(n)}, B^{(n)}]_t)=\frac{1}{n}\sum_{i=1}^{\lf nt\rf}\mb{E}(\varepsilon_i^2/\chi)\leq t.\]
Therefore, by Markov's inequality, for any $K>0$, 
\[\mb{P}([B^{(n)}, B^{(n)}]_t>K)\leq \frac{\mb{E}([B^{(n)}, B^{(n)}]_t)}{K}\leq \frac{t}{K}\to 0, \quad \text{as}\ K\to\infty.\]
This confirms the uniform tightness. The stochastic integral convergence theorem then gives the joint convergence
\[
\left(\int_0^\cdot \widehat H^{(n)}_s \dd B^{(n)}_s,\Lambda^{(n)}\right)
\Rightarrow
\left(\int_0^\cdot \ee^{\Lambda_s}\dd B_s,\Lambda\right)
\quad\text{in } \mathcal{D}([0,1],\mb{R}^2).
\]
Thus, the pair $(X^{(n)}, \Lambda^{(n)})$ converges weakly to $(X, \Lambda)$ as stated.

\section{Proof of Theorem 11}
The proof proceeds in three steps. The first and third are adaptations of
arguments already established in \cite{wang2025rough} and in the proof of Theorem 10, and the second handles the additional correlation structure introduced by the quasi score.

\medskip
\noindent
\textit{Step 1: Baseline analysis.}

The time varying baseline $\omega_t^{(n)}$ sets the deterministic component associated with the nonzero initial value. The rescaled path uses the index $\lfloor ns\rfloor\vee1$ near the origin. Put $\psi_0^{(n)}:=1$ and $\Psi_m^{(n)} := \sum_{k=0}^{m}\psi_k^{(n)}$, with the convention $\Psi_{-1}^{(n)}:=0$. Let
$S_t^{(n)} := \sum_{j=0}^{t-1}\psi_{t-j}^{(n)}\Delta M_j$ denote the
centred stochastic part of the MA representation of $\lambda_t^{(n)}$. Write $A_r^{(n)}:=\sum_{u=1}^{r-1}\phi_u^{(n)}$. Since the baseline recursion begins at $t=1$, its deterministic part is, for $t\ge1$,
\begin{equation}\label{eq:Dt_closed}
\begin{split}
D_t^{(n)}
&:=
\frac{(1-a_n)\theta}{\omega_n}\sum_{k=0}^{t-1}
  \psi_k^{(n)}\,\omega_{t-k}^{(n)}  \\
&=
\theta(1-a_n)\Psi_{t-1}^{(n)}
 + \xi\theta\left\{
 1-(1-a_n)\bigl(\Psi_{t-1}^{(n)}-1\bigr)
 \right\},
\end{split}
\end{equation}
where we used the renewal identity
$\sum_{k=0}^{t-1}\psi_k^{(n)}A_{t-k}^{(n)}=\Psi_{t-1}^{(n)}-1$.
For $s\in[0,1]$, set
$m_n(s)=\lfloor ns\rfloor\vee1$. The kernel asymptotics
$(1-a_n)\Psi_{m_n(s)-1}^{(n)}\to F^{\alpha,\kappa}(s)$
show that the deterministic component converges to
\begin{equation}\label{eq:Dt_limit}
    D_{m_n(s)}^{(n)}
\;\longrightarrow\;
\theta F^{\alpha,\kappa}(s)+\xi\theta\bigl(1-F^{\alpha,\kappa}(s)\bigr).
\end{equation}
The convergence is uniform on $[0,1]$ by the uniform renewal convergence and the continuity of $F^{\alpha,\kappa}$. This is the deterministic part of the rough Volterra limit with initial condition $\Lambda_0=\theta\xi$. Applying the same argument as in Theorem 1 to the stochastic component, with this deterministic component and $\gamma=v$, gives

\begin{equation}\label{eq:Lconv}
\Lambda^{(n)}\Rightarrow \Lambda
\quad\text{in }\mathcal{D}([0,1]).
\end{equation}

\medskip
\noindent
\textit{Step 2: Bivariate covariance structure.}

Define the bivariate partial sum process
$\mathbf{M}_t^{(n)} := n^{-1/2}\sum_{i=1}^{\lfloor nt\rfloor}
(\varepsilon_i, \nabla_i)= (B^{(n)},V^{(n)})$.
The covariance matrix
$\Sigma := \mb{E}[(\varepsilon_i,\nabla_i)^\top(\varepsilon_i,\nabla_i)]$
has entries
\begin{equation}\label{eq:Sigma}
\Sigma_{11} = 1,\qquad
\Sigma_{12} = \mb{E}[\varepsilon_i\nabla_i] = \rho,\qquad
\Sigma_{22} = \mb{E}[\nabla_i^2] = v^2,
\end{equation}
where $\Sigma_{12} = \rho$ follows from $\mb{E}[\varepsilon_i^2] = 1$ and $\mb{E}[\varepsilon_i|\varepsilon_i|] = 0$, and $\Sigma_{22} = v^2$ from $\text{Var}(|\varepsilon_i|) = 1 - 2/\pi$.
The predictable quadratic covariation satisfies $\langle\mathbf{M}^{(n)}\rangle_t = (\lfloor nt\rfloor/n)\,\Sigma \to t\Sigma$ almost surely. Since $\{\varepsilon_i\}$ are i.i.d.\ Gaussian, the Lindeberg condition
holds trivially, and the same argument as in the proof of Theorem 10 applies componentwise. The multidimensional martingale FCLT gives
\begin{equation}\label{eq:BFCLT}
\mathbf{M}^{(n)} = (B^{(n)},V^{(n)})
\;\Rightarrow\; (B,V)
\quad\text{in }\mathcal{D}([0,1],\mathbb{R}^2),
\end{equation}
where $(B,V)$ is a bivariate Brownian motion with covariance matrix
$t\Sigma$. Setting $W := V/v$ yields a standard Brownian motion with
$\langle B,W\rangle_t = \rho t/v = \bar\rho t$.

\medskip
\noindent
\textit{Step 3: Joint convergence.}

The convergence \eqref{eq:BFCLT}, together with the H\"{o}lder tightness argument used in Theorem 1, gives the joint convergence of $B^{(n)}$ and the linearly interpolated score partial sum in the product topology needed for the GFO map. The deterministic component converges uniformly by Step 1. Applying the GFO map to the score coordinate gives the joint convergence of $B^{(n)}$ and the auxiliary volatility process. For the passage from the auxiliary process to $\Lambda^{(n)}$, the estimate in Proposition 9 applies to the same weights, while the deterministic part has already been matched in Step 1. Hence
$(B^{(n)},\Lambda^{(n)})\Rightarrow(B,\Lambda)$ in
$\mathcal{D}([0,1],\mathbb{R}^2)$.

Let $\widehat H^{(n)}$ be the predictable step process defined in the proof of Theorem 10. It gives the exact representation
\[
X_t^{(n)}
=X_0-\frac12\int_0^{\lfloor nt\rfloor/n}
    (\widehat H_s^{(n)})^2\dd s
    +\int_0^t\widehat H_s^{(n)}\dd B_s^{(n)}.
\]
As in Theorem 10, the one-step shift is asymptotically negligible, so $\widehat H^{(n)}\to\ee^\Lambda$ jointly with $(B^{(n)},\Lambda^{(n)})$. The remaining terminal mesh interval in the drift has length at most $1/n$ and is negligible by this joint convergence.

Applying the stochastic integral convergence theorem of \cite{kurtz1991weak} together with the continuous mapping theorem for the Riemann integral map $f\mapsto\int_0^\cdot f(s)\dd s$ gives the joint convergence
\[
(X^{(n)},\Lambda^{(n)})
\;\Rightarrow\;
\left(
X_0 - \frac{1}{2}\int_0^\cdot \ee^{2\Lambda_s}\dd s
     + \int_0^\cdot \ee^{\Lambda_s}\dd B_s,\Lambda
\right)
\quad\text{in }\mathcal{D}([0,1],\mathbb{R}^2).
\]
This is the desired convergence to $(X,\Lambda)$ and completes the proof.

\section{Proof of Lemma 13}

The proof separates $\lambda_{\lfloor ns\rfloor}^{(n)}$ into a deterministic term and a weighted martingale sum. To rule out the possibility that the martingale sum nearly cancels the deterministic term, we first establish the following delocalization result.

\begin{lemma}[Weighted martingale delocalization]\label{lem:weighted_delocalization_supp}
Let $\{\xi_{n,j}\}$ be a martingale difference array with respect to $\{\mathcal F_{n,j}\}$. Suppose that, for constants $0<\underline u\le\overline u<\infty$,
\[
    \underline u^2
    \le
    \mb E[\xi_{n,j}^2\mid\mathcal F_{n,j-1}]
    \le
    \overline u^2
\]
almost surely, and that $\sup_{n,j}\mb E|\xi_{n,j}|^8<\infty$. Let $\{w_{n,j}\}$ be deterministic weights, with only finitely many nonzero terms for each $n$, such that
\[
    \sum_jw_{n,j}^2\to q\in(0,\infty),
    \qquad
    \max_j|w_{n,j}|\to0.
\]
Then, for every deterministic sequence $r_n\downarrow0$,
\[
    \sup_{z\in\mb R}
    \mb P\left(
    \left|
    \sum_jw_{n,j}\xi_{n,j}-z
    \right|\le r_n
    \right)
    \to0.
\]
\end{lemma}

\begin{proof}
We use Theorem 1 of \citet{gurel2014localization}. In the form needed here, it states that for any $\delta\in(0,1]$ and $\beta<1/2$ there exist constants $C,\eta>0$ such that, if a martingale with $N$ increments has conditional variances between $\delta$ and $1$ and increments bounded by $N^\beta$, then
\[
    \sup_{z\in\mb R}
    \mb P(|M_N-z|\le N^\beta)
    \le
    CN^{-\eta}.
\]
We now reduce the weighted sum to this theorem.

Set
\[
    q_n=\sum_j w_{n,j}^2,\qquad
    \varrho_n=\max_j|w_{n,j}|,\qquad
    K_n=\left\lfloor\frac1{\varrho_n+r_n}\right\rfloor,
    \qquad
    h_n=\frac{q_n}{K_n}.
\]
Then $K_n\to\infty$, $K_n\varrho_n^2\to0$, $K_nr_n^2\to0$, and $\varrho_n^2/h_n\to0$. Partition the nonzero indices into $K_n$ consecutive deterministic blocks. More precisely, let $Q_{n,\ell}=\sum_{j\le \ell}w_{n,j}^2$ and, for $m=1,\ldots,K_n-1$, let $\tau_{n,m}$ be the first index for which $Q_{n,\ell}\ge mh_n$; put $\tau_{n,0}$ one index before the first nonzero weight and let $\tau_{n,K_n}$ be the last nonzero index. With $I_{n,m}=\{\tau_{n,m-1}+1,\ldots,\tau_{n,m}\}$ and $H_{n,m}=\sum_{j\in I_{n,m}}w_{n,j}^2$, the maximal weight condition gives, for all sufficiently large $n$,
\[
    \frac{h_n}{2}
    \le
    H_{n,m}
    \le
    \frac{3h_n}{2},
    \qquad m=1,\ldots,K_n.
\]

Define
\[
    B_{n,m}:=\sum_{j\in I_{n,m}}w_{n,j}\xi_{n,j},
\]
and let $\mathcal G_{n,m}=\mathcal F_{n,\tau_{n,m}}$ be the filtration at the end of the $m$-th block. Then $(B_{n,m},\mathcal G_{n,m})$ is a martingale-difference array. Conditional orthogonality and the variance bounds give
\[
    \underline u^2H_{n,m}
    \le
    \mb E[B_{n,m}^2\mid\mathcal G_{n,m-1}]
    \le
    \overline u^2H_{n,m}.
\]
Let
\[
    d_n=\overline u\sqrt{\frac{3h_n}{2}},
    \qquad
    X_{n,m}=\frac{B_{n,m}}{d_n}.
\]
Then, with $\delta=\underline u^2/(3\overline u^2)$,
\begin{equation}\label{block_variance_supp}
    \delta
    \le
    \mb E[X_{n,m}^2\mid\mathcal G_{n,m-1}]
    \le
    1.
\end{equation}
The martingale Rosenthal inequality and the uniform eighth moment bound imply
\[
\begin{split}
    \mb E|B_{n,m}|^8
    &\le
    C\,\mb E\left(
    \sum_{j\in I_{n,m}}w_{n,j}^2
    \mb E[\xi_{n,j}^2\mid\mathcal F_{n,j-1}]
    \right)^4
    +
    C\sum_{j\in I_{n,m}}w_{n,j}^8\mb E|\xi_{n,j}|^8  \\
    &\le
    CH_{n,m}^4+C\varrho_n^6H_{n,m}.
\end{split}
\]
Since $H_{n,m}\asymp h_n$, $\varrho_n^2/h_n\to0$, and $d_n^2\asymp h_n$,
\begin{equation}\label{block_eighth_supp}
    \sup_{n,m}\mb E|X_{n,m}|^8<\infty.
\end{equation}

Choose $\beta=1/4$ and set $L_n=K_n^\beta/2$. Define
\[
    T_{n,m}=X_{n,m}\mathbf 1_{\{|X_{n,m}|\le L_n\}},
    \qquad
    Z_{n,m}=T_{n,m}
    -\mb E[T_{n,m}\mid\mathcal G_{n,m-1}].
\]
Let
\[
    R_{n,m}=
    \mb E\!\left[
    X_{n,m}^2\mathbf 1_{\{|X_{n,m}|>L_n\}}
    \mid\mathcal G_{n,m-1}
    \right],
    \qquad
    A_{n,m}=\{R_{n,m}>\delta/4\}.
\]
By \eqref{block_variance_supp} and conditional Cauchy--Schwarz, on $A_{n,m}^c$,
\[
    \frac{\delta}{2}
    \le
    \operatorname{Var}(Z_{n,m}\mid\mathcal G_{n,m-1})
    \le
    1,
    \qquad
    |Z_{n,m}|\le K_n^\beta.
\]
On $A_{n,m}$, enlarge the probability space and replace $Z_{n,m}$ by an independent centered two-point variable with variance $\delta/2$. The resulting martingale differences $\widetilde X_{n,m}$, with respect to the enlarged block filtration, satisfy
\[
    \frac{\delta}{2}
    \le
    \mb E[\widetilde X_{n,m}^2
    \mid\widetilde{\mathcal G}_{n,m-1}]
    \le
    1,
    \qquad
    |\widetilde X_{n,m}|\le K_n^\beta.
\]
The theorem of \citet{gurel2014localization}, applied with $N=K_n$, gives constants $C,\eta>0$ such that
\begin{equation}\label{blocked_deloc_supp}
    \sup_{y\in\mb R}
    \mb P\left(
    \left|
    \sum_{m=1}^{K_n}\widetilde X_{n,m}-y
    \right|
    \le K_n^\beta
    \right)
    \le CK_n^{-\eta}.
\end{equation}

It remains to compare the modified and original block sums. From \eqref{block_eighth_supp},
\[
    \mb P\left(\bigcup_{m=1}^{K_n}A_{n,m}\right)
    \le
    C K_nL_n^{-6}=O(K_n^{-1/2}),
    \qquad
    \mb P\left(\max_{1\le m\le K_n}|X_{n,m}|>L_n\right)
    \le
    C K_nL_n^{-8}=O(K_n^{-1}),
\]
and
\[
    \mb E\sum_{m=1}^{K_n}
    \left|
    \mb E\!\left[
    X_{n,m}\mathbf 1_{\{|X_{n,m}|>L_n\}}
    \mid\mathcal G_{n,m-1}
    \right]
    \right|
    \le
    C K_nL_n^{-7}=O(K_n^{-3/4}).
\]
Hence
\[
    \sum_{m=1}^{K_n}\widetilde X_{n,m}
    -
    \sum_{m=1}^{K_n}X_{n,m}
    \toP0.
\]
Finally, $d_n\asymp K_n^{-1/2}$ and $r_n/d_n\le Cr_n\sqrt{K_n}\to0$. Since
\[
    \sum_jw_{n,j}\xi_{n,j}
    =
    d_n\sum_{m=1}^{K_n}X_{n,m},
\]
we have, uniformly in $z\in\mb R$,
\[
\begin{split}
    &\mb P\left(
    \left|
    \sum_jw_{n,j}\xi_{n,j}-z
    \right|\le r_n
    \right)  \\
    &\qquad\le
    \mb P\left(
    \left|
    \sum_{m=1}^{K_n}\widetilde X_{n,m}
    -
    \frac{z}{d_n}
    \right|
    \le
    K_n^\beta
    \right)+o(1).
\end{split}
\]
Taking the supremum over $z$ and using \eqref{blocked_deloc_supp} proves the claim.
\end{proof}

By Assumption 1, there exist constants $0<\underline u\le \overline u<\infty$ such that $\underline u\le U(x)\le\overline u$ for all $x\in\mb R$.
Fix $s\in(0,1]$ and $M>0$, and put $t=\lfloor ns\rfloor$. From the moving average representation,
\[
    \lambda_t^{(n)}
    =
    D_n+J_n,
    \qquad
    D_n:=\omega_n\left(1+\sum_{i=1}^{t}\psi_i^{(n)}\right),
    \qquad
    J_n:=\sum_{j=0}^{t-1}\psi_{t-j}^{(n)}\Delta M_j^{(n)} .
\]
The uniform convergence of the renewal distribution functions gives
\[
    \frac{1-a_n}{a_n}\sum_{i=1}^{t}\psi_i^{(n)}
    \to
    F^{\alpha,\kappa}(s)>0.
\]
By Assumption 2,
\[
    n^{1/2-\alpha}D_n
    \to
    c_s:=
    \frac{\alpha\theta}{K\Gamma(1-\alpha)\kappa\mu}
    F^{\alpha,\kappa}(s),
\]
and $c_s\ne0$. Normalize by this rate and define
\[
    \widehat\lambda_n:=n^{1/2-\alpha}\lambda_t^{(n)},\qquad
    \widehat D_n:=n^{1/2-\alpha}D_n,\qquad
    \widehat J_n:=n^{1/2-\alpha}J_n.
\]
Then $\widehat D_n\to c_s$.

The martingale part can be written as
\[
    \widehat J_n
    =
    \sum_{j=0}^{t-1}w_{n,j}\Delta M_j^{(n)},
    \qquad
    w_{n,j}:=n^{1/2-\alpha}\psi_{t-j}^{(n)} .
\]
It remains to verify the deterministic weight conditions in Lemma \ref{lem:weighted_delocalization_supp}. Put
\[
    R_{n,k}:=n(1-a_n)\psi_k^{(n)},\qquad
    R_n(x):=\sum_{k\ge1}R_{n,k}\mathbf1_{((k-1)/n,k/n]}(x),
    \qquad
    \ell_n:=n^\alpha(1-a_n).
\]
Since $R_n=a_n\rho_n$, the $L^2$ convergence of the histogram density gives $R_n\to f^{\alpha,\kappa}$ in $L^2(0,1)$. Assumption 2 gives $\ell_n\to K\Gamma(1-\alpha)\kappa/\alpha$, and
\[
    w_{n,j}
    =
    \frac{R_{n,t-j}}{\sqrt n\,\ell_n}.
\]
Therefore
\[
    \sum_{j=0}^{t-1}w_{n,j}^2
    =
    \frac1{\ell_n^2}
    \int_0^{t/n}R_n(x)^2\dd x
    \to
    \left(\frac{\alpha}{K\Gamma(1-\alpha)\kappa}\right)^2
    \int_0^s f^{\alpha,\kappa}(x)^2\dd x
    \in(0,\infty).
\]
Here the convergence follows from $t/n\to s$ and $\|R_n^2-(f^{\alpha,\kappa})^2\|_1\le\|R_n-f^{\alpha,\kappa}\|_2(\|R_n\|_2+\|f^{\alpha,\kappa}\|_2)$. The integral is finite because $\alpha>1/2$.

It remains to check the maximal weight. Setting $R_{n,0}=0$, the discrete Parseval estimate used in Section S4, with lag $d=1$, gives
\[
    \sum_{k\ge0}(R_{n,k+1}-R_{n,k})^2
    \le
    Cn^{2-2\alpha}.
\]
Also, by the same $L^2$ convergence,
\[
    \sum_{k=1}^{n}R_{n,k}^2
    =
    n\int_0^1R_n(x)^2\dd x
    \le Cn.
\]
Thus, for $1\le k\le t\le n$,
\[
\begin{split}
    R_{n,k}^2
    &=
    \sum_{\ell=0}^{k-1}
    (R_{n,\ell+1}-R_{n,\ell})
    (R_{n,\ell+1}+R_{n,\ell})\\
    &\le
    \left[\sum_{\ell\ge0}(R_{n,\ell+1}-R_{n,\ell})^2\right]^{1/2}
    \left[\sum_{\ell=0}^{n-1}(R_{n,\ell+1}+R_{n,\ell})^2\right]^{1/2}
    \le Cn^{3/2-\alpha}.
\end{split}
\]
Consequently,
\[
    \max_{0\le j<t}|w_{n,j}|^2
    \le
    \frac{Cn^{3/2-\alpha}}{n\ell_n^2}
    \le
    Cn^{1/2-\alpha}
    \to0.
\]
Assumption 1 gives conditional variances bounded away from zero and infinity, and Assumption 3 gives the required eighth moment bound. Hence Lemma \ref{lem:weighted_delocalization_supp} yields, for $r_n:=Mn^{1/2-\alpha}\downarrow0$,
\[
\begin{split}
    \mb P\left(|\lambda_t^{(n)}|\le M\right)
    &=
    \mb P\left(|\widehat D_n+\widehat J_n|\le r_n\right)\\
    &\le
    \sup_{z\in\mb R}
    \mb P\left(|\widehat J_n-z|\le r_n\right)
    \to0.
\end{split}
\]

\section{Proof of Proposition 12}

The main theorem for the leveraged rough volatility example gives the weak convergence of the log price process, and the limiting log price is continuous. Convergence in the Skorokhod topology to a continuous limit is equivalent to uniform convergence, and the simulation scheme uses linear interpolation of the discrete log price. Hence,
\[
    S^{(n)}=\exp(X^{(n)})\Rightarrow S=\exp(X)
    \quad\text{in }\mathcal C([0,1]).
\]
For any payoff functional $\Phi_n$ converging to $\Phi$ at the limiting path, the extended continuous mapping theorem gives
\[
    \Phi_n(S^{(n)})\Rightarrow \Phi(S).
\]
It remains only to justify convergence of expectations for the unbounded payoffs.

For the discrete log price recursion, the Gaussian return normalization gives
\[
    \mb E\!\left[
        \exp\!\left(
            -\frac12\sigma_{k,n}^2/n
            +\sigma_{k,n}\varepsilon_k/\sqrt n
        \right)
        \middle| \mathcal F_{k-1}^{(n)}
    \right]=1,
\]
where $\sigma_{k,n}=\left[\exp(\lambda_k^{(n)})\right]^{(1-a_n)\theta/\omega_n}$. Thus the discrete stock price is a martingale and $\mb E[S_1^{(n)}]=S_0$. For the limiting model,
\[
    S_t
    =
    S_0\exp\left(
        -\frac12\int_0^t \ee^{2\Lambda_s}\dd s
        +
        \int_0^t \ee^{\Lambda_s}\dd B_s
    \right).
\]
Since the volatility is lognormal Volterra with Gaussian driver and the effective correlation between the stock Brownian motion and the log volatility driver is nonpositive, $\mu\bar\rho\le0$, Theorem 1.1(1) of \citet{gassiat2019martingale} applies and yields that $S$ is a true martingale on the finite horizon considered here. Therefore $\mb E[S_1]=S_0$, and the convergence $S_1^{(n)}\Rightarrow S_1$ together with convergence of first moments gives uniform integrability of $\{S_1^{(n)}\}_{n\ge1}$. European call prices then converge, and European put prices are bounded by $K$.

The same argument applies to Asian calls. Indeed, the grid average has expectation $S_0$, and the limiting time average has expectation $S_0$ by Fubini and the true martingale property of $S$. The convergence of paths gives convergence of the averages, and convergence of their first moments gives uniform integrability. Asian put prices are bounded.

For the lookback call, we need the following maximal uniform integrability result.

\begin{lemma}\label{lem:lookback_max_ui}
Under the assumptions of Proposition 12, the family
\[
    \left\{\max_{0\le k\le n}S_{k/n}^{(n)}:n\ge1\right\}
\]
is uniformly integrable.
\end{lemma}

\begin{proof}
Put $\Lambda_{k,n}=c_n\lambda_k^{(n)}$, $u_{k,n}=n^{-1/2}\ee^{\Lambda_{k,n}}$, and
\[
    M_{k,n}:=\frac{S_{k/n}^{(n)}}{S_0}
    =\prod_{i=1}^k
    \exp\left(u_{i,n}\varepsilon_i-\frac12u_{i,n}^2\right).
\]
Since $u_{i,n}$ is $\mathcal F_{i-1}$ measurable, $M_{\cdot,n}$ is a positive martingale. Write
\[
    h(x)=\rho x+\zeta\left(|x|-\sqrt{2/\pi}\right).
\]
The moving average representation used in the proof of Theorem 11 gives
\[
    \Lambda_{k,n}
    =d_{k,n}+\sum_{j=0}^{k-1}b_{k-j,n}h(\varepsilon_j),
    \qquad b_{\ell,n}=c_n\psi_\ell^{(n)},
\]
where $d_{k,n}$ contains the deterministic baseline and any fixed initial contribution. The uniform baseline convergence in \eqref{eq:Dt_limit}, the $L^2$ renewal kernel convergence in Proposition 4, and $\sum_{\ell\ge1}\psi_\ell^{(n)}=a_n/(1-a_n)$ imply
\begin{equation}\label{eq:lookback_kernel_bounds}
    \sup_{n,k}|d_{k,n}|<\infty,
    \qquad
    \sup_n\sum_{\ell=1}^n b_{\ell,n}^2<\infty,
    \qquad
    \sum_{\ell=1}^n|b_{\ell,n}|\le C\sqrt n.
\end{equation}
Indeed, if $\rho_n$ is the histogram density used in Proposition 4, then
\[
    \sum_{\ell=1}^n b_{\ell,n}^2
    =\frac{a_n^2\theta^2}{n\omega_n^2}
      \|\rho_n\|_{L^2([0,1])}^2,
    \qquad
    \sum_{\ell=1}^n|b_{\ell,n}|
    \le a_n\left|\frac{\theta}{\omega_n}\right|,
\]
and the bounds follow from Assumption 2.

Let $r_n=n^{-3/8}$ and stop before the first increment for which $u_{i,n}>r_n$:
\[
    \tau_n=\inf\{1\le i\le n:u_{i,n}>r_n\},
    \qquad \inf\varnothing=n+1,
    \qquad I_{i,n}=\mathbf 1_{\{i<\tau_n\}}.
\]
The event $\{\tau_n=i\}$ belongs to $\mathcal F_{i-1}$. Hence
\[
    \widehat M_{k,n}
    :=\prod_{i=1}^k
      \exp\left[I_{i,n}\left(u_{i,n}\varepsilon_i-\frac12u_{i,n}^2\right)\right]
    =M_{k\wedge(\tau_n-1),n}
\]
is again a positive martingale. Define the stopped stock measure by
\[
    \frac{\dd\widehat{\mb Q}_n}{\dd\mb P}=\widehat M_{n,n}.
\]
Write $\widehat{\mb E}_n$ for expectation under $\widehat{\mb Q}_n$.
Under $\widehat{\mb Q}_n$, the innovations satisfy
\[
    \varepsilon_i=Z_i+I_{i,n}u_{i,n},
    \qquad Z_i\mid\mathcal F_{i-1}\sim\mathcal N(0,1),
\]
while the pre-sample innovation $\varepsilon_0$ is unchanged.

For $q\ge0$, let $m(q)=\mb E[h(Z+q)]$. Then
\[
    m(q)=\rho q+\zeta g(q),
    \qquad
    0\le g(q):=\mb E|Z+q|-\mb E|Z|
    \le\frac{q^2}{\sqrt{2\pi}}.
\]
For the pre-sample innovation, set $I_{0,n}u_{0,n}=0$.
Since finitely many values of $n$ do not affect uniform integrability, we henceforth take $n$ sufficiently large. The coefficients $b_{\ell,n}$ then have the sign of $\mu$. Since $\mu\bar\rho\le0$ and $v>0$, we have $\mu\rho\le0$, and therefore
\begin{equation}\label{eq:lookback_drift_bound}
    b_{\ell,n}m(q)\le C|b_{\ell,n}|q^2.
\end{equation}
Before stopping, $I_{i,n}u_{i,n}\le r_n$, so \eqref{eq:lookback_kernel_bounds} and \eqref{eq:lookback_drift_bound} give
\[
    \sum_{j=0}^{k-1}b_{k-j,n}m(I_{j,n}u_{j,n})
    \le Cr_n^2\sum_{\ell=1}^n|b_{\ell,n}|
    \le Cn^{-1/4}.
\]
Centering the score terms, set
\[
    Y_{k,n}:=\sum_{j=0}^{k-1}b_{k-j,n}
    \left\{h(Z_j+I_{j,n}u_{j,n})-m(I_{j,n}u_{j,n})\right\}.
\]
Then $\Lambda_{k,n}$ is the sum of $d_{k,n}$, the preceding drift term, and $Y_{k,n}$.

The function $h$ is Lipschitz with constant $L_h=|\rho|+|\zeta|$. The Gaussian concentration inequality therefore yields
\[
    \widehat{\mb E}_n\left[
      \exp\left(t\{h(Z_j+I_{j,n}u_{j,n})-m(I_{j,n}u_{j,n})\}\right)
      \middle|\mathcal F_{j-1}\right]
    \le \exp\left(\frac12L_h^2t^2\right).
\]
For $j=0$, the same estimate holds without conditioning.
Iterating conditional expectations and using the $\ell^2$ bound in \eqref{eq:lookback_kernel_bounds} gives constants $C,c>0$, independent of $k$ and $n$, such that
\begin{equation}\label{eq:lookback_gaussian_tail}
    \widehat{\mb Q}_n(\Lambda_{k,n}>x)
    \le C\exp\left[-c(x-C)_+^2\right].
\end{equation}
The stopping event requires
\[
    \Lambda_{k,n}>\log(\sqrt n\,r_n)=\frac18\log n
\]
for some $k\le n$. Consequently, \eqref{eq:lookback_gaussian_tail} and a union bound show that, for every $m>0$,
\begin{equation}\label{eq:lookback_stop_prob}
    \widehat{\mb Q}_n(\tau_n\le n)=o(n^{-m}).
\end{equation}
The same conditional moment generating function estimate also gives
\begin{equation}\label{eq:lookback_exp_moments}
    \sup_n\max_{1\le k\le n}
    \widehat{\mb E}_n[\ee^{q\Lambda_{k,n}}]<\infty,
    \qquad q>0.
\end{equation}

Set
\[
    A_{k,n}=\sum_{i=1}^kI_{i,n}u_{i,n}^2,
    \qquad
    N_{k,n}=\sum_{i=1}^kI_{i,n}u_{i,n}Z_i.
\]
Under $\widehat{\mb Q}_n$,
\[
    \log\widehat M_{k,n}=\frac12A_{k,n}+N_{k,n}.
\]
By \eqref{eq:lookback_exp_moments}, Jensen's inequality, and Doob's $L^2$ inequality,
\[
    \widehat{\mb E}_n[A_{n,n}^2]
    \le\frac1n\sum_{i=1}^n
       \widehat{\mb E}_n[\ee^{4\Lambda_{i,n}}]\le C,
    \qquad
    \widehat{\mb E}_n\left[\max_{k\le n}|N_{k,n}|^2\right]
    \le4\widehat{\mb E}_n[A_{n,n}]\le C.
\]
Thus, with $\widehat M_n^*=\max_{k\le n}\widehat M_{k,n}$,
\begin{equation}\label{eq:lookback_log_max}
    \sup_n\widehat{\mb E}_n
    \left[(\log^+\widehat M_n^*)^2\right]<\infty.
\end{equation}

For $y>1$, let $\sigma_y$ be the first time at which $\widehat M_{k,n}\ge y$, and put $A_y=\{\widehat M_n^*\ge y\}$. The definition of $\widehat{\mb Q}_n$ and optional sampling for the $\mb P$ martingale $\widehat M_{\cdot,n}$ give
\[
    \widehat{\mb Q}_n(A_y)
    =\mb E[\widehat M_{n,n}\mathbf 1_{A_y}]
    =\mb E[\widehat M_{\sigma_y,n}\mathbf 1_{A_y}]
    \ge y\mb P(A_y).
\]
Consequently,
\[
    \mb P(\widehat M_n^*\ge y)
    \le\frac1y\widehat{\mb Q}_n
       (\log^+\widehat M_n^*\ge\log y).
\]
Using the layer cake formula and \eqref{eq:lookback_log_max}, for $K>\ee$,
\[
    \mb E[\widehat M_n^*;\widehat M_n^*>K]
    \le\frac{C}{(\log K)^2}
       +C\int_{\log K}^{\infty}\frac{\dd x}{x^2}
    \le\frac{C}{\log K}.
\]
Hence $\{\widehat M_n^*:n\ge1\}$ is uniformly integrable.

It remains to remove the stopping. On $\{\tau_n=j\}$, we have $\widehat M_{n,n}=M_{j-1,n}$. The martingale property of $M_{\cdot,n}$ and $\{\tau_n=j\}\in\mathcal F_{j-1}$ therefore give
\[
\begin{split}
    \mb E[M_n^*-\widehat M_n^*]
    &\le\sum_{j=1}^n\sum_{k=j}^n
       \mb E[M_{k,n}\mathbf 1_{\{\tau_n=j\}}]\\
    &=\sum_{j=1}^n(n-j+1)
       \widehat{\mb Q}_n(\tau_n=j)\\
    &\le n\widehat{\mb Q}_n(\tau_n\le n)\longrightarrow0
\end{split}
\]
by \eqref{eq:lookback_stop_prob}. The difference $M_n^*-\widehat M_n^*$ therefore converges to zero in $L^1$, and the preceding uniform integrability result yields the uniform integrability of $\{M_n^*:n\ge1\}$. Multiplication by $S_0$ proves the claim.
\end{proof}

Lookback puts are bounded by $K$. The lookback call payoff is dominated by the running maximum, which is uniformly integrable by Lemma \ref{lem:lookback_max_ui}. Together with convergence of the grid maximum to the running maximum, this proves convergence of the lookback option prices.

For barrier discontinuities, suppose $B_u>S_0>B_d>0$ and $\bar\rho\in(-1,1)$. Write
$B=\bar\rho W+\sqrt{1-\bar\rho^2}\,W^\perp$, where $W^\perp$ is independent of $W$.
Conditionally on $W$, the volatility path is fixed and the log price has the form
\[
    X_t=a_W(t)+\sqrt{1-\bar\rho^2}
    \int_0^t \ee^{\Lambda_s(W)}\,\dd W_s^\perp,
\]
where $a_W$ is a continuous deterministic path under the conditional law. The last
term is a Brownian motion run at the deterministic conditional clock
\[
    A_W(t)=(1-\bar\rho^2)\int_0^t\ee^{2\Lambda_s(W)}\,\dd s,
\]
which is continuous and strictly increasing on every nontrivial interval. Thus, conditionally on $W$, $X$ is a nondegenerate continuous Gaussian process with deterministic continuous mean. By the Gaussian supremum result recalled in \cite[p.~267]{bobkov2008maximum}, the distribution of the supremum of a Gaussian process is absolutely continuous above its take-off point, with any possible atom confined to that point.

It remains to identify the take-off point under the conditional law. The strictly increasing time change gives the conditional support $a_W+C_0([0,1])$. Since $a_W(0)=X_0=\log S_0$, every path in this support starts from $X_0$, while the constant path equal to $X_0$ is also in the support. Hence the take-off point of $\sup_{0\le t\le1}X_t$ is $X_0$. Because $\log B_u>X_0$, the conditional probability of $\sup_{0\le t\le1}X_t=\log B_u$ is zero. Applying the same argument to $-X$ gives zero conditional probability of $\inf_{0\le t\le1}X_t=\log B_d$. Taking expectations over $W$ yields
\[
    \mb P\left(\sup_{0\le t\le 1}S_t=B_u\right)=0,
    \qquad
    \mb P\left(\inf_{0\le t\le 1}S_t=B_d\right)=0.
\]
These no atom conditions are the continuity requirements for the barrier payoffs. The up and in call is dominated by the European call payoff, and the down and out put is bounded by $K$. Applying the convergence of expectations above proves the proposition.

\endgroup
\clearpage
\setlength{\bibsep}{3.5pt}
\bibliographystyle{elsarticle-harv}
\bibliography{ref}

\end{document}